\documentclass{amsart}
\usepackage{latexsym}
\usepackage{amsmath}
\usepackage{amssymb}
\usepackage[all]{xy}

\newtheorem{thm}{Theorem}[section]
\newtheorem{prop}[thm]{Proposition}

\newtheorem{cor}[thm]{Corollary}
\newtheorem{lem}[thm]{Lemma}
\theoremstyle{definition}
\newtheorem{defn}[thm]{Definition}

\newtheorem{claim}[thm]{Claim}
\theoremstyle{remark}
\newtheorem{rem}[thm]{Remark}
\newtheorem{ex}[thm]{Example}

\newcommand{\K}{{\mathbb K}}
\newcommand{\G}{{\mathcal G}}
\newcommand{\C}{{\mathcal C}}
\newcommand{\E}{{\mathcal E}}
\newcommand{\calH}{{\mathcal H}}
\newcommand{\calK}{{\mathcal K}}
\newcommand{\e}{\varepsilon}
\newcommand{\D}{\mathcal{D}}
\newcommand{\mapright}[1]{%
 \smash{\mathop{%
  \hbox to 1cm{\rightarrowfill}}\limits_{#1}}}
\newcommand{\maprightd}[2]{%
 \smash{\mathop{%
  \hbox to 1.2cm{\rightarrowfill}}\limits^{#1}\limits_{#2}}}
\newcommand{\mapleft}[1]{%
 \smash{\mathop{%
  \hbox to 1cm{\leftarrowfill}}\limits_{#1}}}
\newcommand{\mapleftu}[1]{%
 \smash{\mathop{%
  \hbox to 0.8cm{\leftarrowfill}}\limits^{#1}}}
\newcommand{\maprightu}[1]{%
 \smash{\mathop{%
  \hbox to 1cm{\rightarrowfill}}\limits^{#1}}}
\newcommand{\maprightud}[2]{%
 \smash{\mathop{%
  \hbox to 1cm{\rightarrowfill}}\limits^{#1}_{#2}}}
\newcommand{\mapleftud}[2]{%
 \smash{\mathop{%
  \hbox to 1cm{\leftarrowfill}}\limits^{#1}_{#2}}}

\newcounter{eqn}[section]

\def\theeqn{\textnormal{(\thesection.\arabic{eqn})}}

\def\eqnlabel#1{%
  \refstepcounter{eqn}%
  \label{#1}%
  \leqno{\theeqn}}

\begin{document}

\title[Association schemoids and their categories]{
Association schemoids and their categories \\
}

\footnote[0]{

This research was partially supported by a Grant-in-Aid for Scientific
Research HOUGA 25610002
from Japan Society for the Promotion of Science.

Department of Mathematical Sciences, 
Faculty of Science,  
Shinshu University,   
Matsumoto, Nagano 390-8621, Japan   
e-mail:{\tt kuri@math.shinshu-u.ac.jp} 

Department of Mathematical Sciences, 
Faculty of Science,  
Shinshu University,   
Matsumoto, Nagano 390-8621, Japan  
e-mail:{\tt kmatsuo@math.shinshu-u.ac.jp}
}

\author{Katsuhiko KURIBAYASHI and Kentaro MATSUO}
\date{}
   
\maketitle


\begin{abstract}
We propose the notion of {\it association schemoids} generalizing that of association schemes 
from small categorical points of view. 
In particular, a generalization of the Bose-Mesner algebra of an association scheme appears as a subalgebra in the category algebra of the underlying category of a schemoid. 
In this paper, the equivalence between the categories of groupoids and that of thin association schemoids is established. 
Moreover linear extensions of schemoids are considered. 
A general theory of the Baues-Wirsching cohomology deduces a classification theorem for such extensions of a schemoid.  
We also introduce two relevant categories of schemoids into which the categories of schemes due to Hanaki 
and due to French are embedded, respectively.   
\end{abstract}

\section{Introduction}
Finite groups are investigated in appropriate derived categories via group rings 
with categorical representation theory and in the category of topological spaces via classifying spaces 
with homotopy theory. Since association schemes are 
regarded as generalizations of finite groups, it is natural to construct a categorical framework for studying 
such generalized groups. Then we introduce {\it association schemoids} and their categories 
in expectation of interaction with association schemes, groupoids, the classifying spaces of small categories 
and the new notion in the study of these subjects. 

An association scheme is a pair of a finite set and a particular partition of the Cartesian square 
of the set. The notion plays a crucial role 
in algebraic combinatorics \cite{B-I}, including the study of designs and graphs, and in coding theory \cite{D}.  
In fact, such schemes encode combinatorial phenomena 
in terms of representation theory of finite dimensional algebras. To this end, 
we may use the Bose-Mesner algebra of an association scheme which, by definition, 
is the matrix algebra generated by adjacency matrices of the elements of the partition.  
Each spin model \cite{Jones}, which is a square matrix yielding an invariant of links and knots, is realized 
as an element of the Bose-Mesner algebra of some association scheme \cite{Jaeger, J-M-N, N}. 
This also exhibits the importance of association schemes. Moreover, the structure theory   
of association schemes have been investigated in the framework of group theory as generalized groups; 
see \cite{Z_1995, Z_book}. 
Very recently, global nature of the interesting objects is studied 
in such a way as to construct categories consisting of finite association schemes 
and appropriate morphisms \cite{F, H}.  
Interaction with the above-mentioned subjects 
makes the realm of such schemes more fruitful.

In this paper, by generalizing the notion of association schemes itself from a categorical point of view, 
we introduce a particular structure on a small category 
and coin the notion of {\it association schemoids}. 
Roughly speaking, a specific partition of the set of morphisms brings the additional structure into the small category we deal with.  
One of important points is that the Bose-Mesner algebra 
associated with a schemoid can be defined in a natural way as a subalgebra in the category algebra of 
the underlying category of the given schemoid.  Here the category algebra is a generalization of the path algebra 
associated with a quiver, 
which is a main subject of consideration in representation theory of associative algebras \cite{ASS}. 
Moreover, we should mention that the category ${\bf AS}$ of finite association schemes introduced by Hanaki  \cite{H} 
is imbedded into our category ${\bf ASmd}$ of association schemoids fully and faithfully; see Theorem \ref{thm:functors}.

A thin association scheme is identified with a group; see \cite[(1.12)]{Z_1995} for example. With our setting, 
the correspondence is generalized; that is, 
we give an equivalence between the category of based {\it thin} association schemoids and that of groupoids; 
see Theorem \ref{thm:eq_categories}. Indeed, the equivalence is 
an expected lift of a functor from the category of finite groups to that of 
based thin association schemes in \cite{H}; see the diagram (6.1) below and the ensuring comments. 
 
Baues and Wirsching \cite{B-W} have defined the linear extension of a small category, which is a generalization of a group extension, and 
have proved a classification theorem for such extensions with cohomology of small categories. 
We show that each linear extension of a given schemoid admits a unique 
schemoid structure; see Proposition \ref{prop:extension}. 
This result enables one to conclude that extensions of a schemoid are also classified by the Baues-Wirsching cohomology; 
see Theorem \ref{classification_ex}. In our context, every extension of 
an arbitrary association scheme is trivial; see Corollary \ref{cor:trivial_extensions}. 
Unfortunately, our extensions of a schemoid do {\it not} cover extensions of an association scheme, 
which are investigated in \cite{Bang-Hirasaka}  and \cite{Hida}.  

In \cite{F}, French has introduced a wide subcategory of the category ${\bf AS}$ of finite association schemes. 
The subcategory consists of all finite association schemes and particular maps 
called the admissible morphisms.  
In particular, the result \cite[Corollary 6.6]{F} asserts that the correspondence sending a finite scheme to its 
Bose-Mesner algebra gives rise to a functor ${\mathcal A}(\text{-})$ from the wide subcategory to the category 
of algebras. To understand the functor in terms of schemoids, 
we introduce a category ${\mathcal B}$ of {\it basic} schemoids and admissible morphisms,  
into which the subcategory due to French is embedded.  
In addition,  the functor ${\mathcal A}(\text{-})$ can be lifted to the category ${\mathcal B}$; see the diagram (6.1) again. 

Let $q: \E \to \C$ be a linear extension over a schemoid $\C$; see Definition \ref{defn:extensions}. 
It is remarkable that 
in some case, the projection $q$ from the schemoid $\E$ to $\C$ is admissible. 
Moreover, the morphism induces 
an isomorphism between the Bose-Mesner algebras of $\E$ and $\C$ 
even if  $\E$ and $\C$ are not equivalent as a category; see Corollary \ref{cor:isomorphism} and Remark \ref{rem:iso_but_not_eqivalent}. 

The plan of  this paper is as follows. In Section 2, we introduce the association schemoids, their Bose-Mesner algebras,  
and their Terwilliger algebras. Ad hoc examples and the category ${\bf ASmd}$ of schemoids mentioned above are also described.  
Section 3 relates the category ${\bf ASmd}$ 
with other categories, especially ${\bf AS}$ and the category of groupoids. 
In Section 4, after dealing with (semi-)thin 
schemoids, we prove Theorem \ref{thm:eq_categories}. 
Section 5 explores linear extensions of schemoids. 
At the end of the section, we give an example of a non-split schemoid extension.  
Section 6 is devoted to describing some of results due to French \cite{F} in our context, 
namely in terms of schemoids. 
Section 7 explains a way to construct a (quasi-)schemoid thickening a given association scheme. 
In Appendix, we try to explain that a toy model for a network seems to be a schemoid.  
In this paper, we do not pursue properties of the Bose-Mesner algebra and Terwilliger 
algebras of schemoids while one might expect the study of such algebras from categorical representation theory points 
of view. Though we shall need a generalization of the {\it closed subsets} of association schemes when defining subobjects, quotients, limits and colimits in the context of association schemoids, this article does not address the issue.    

As mentioned in \cite{P-Z} by Ponomarenko and Zieschang, 
association schemes are investigated from three different points of view: as algebras,  
purely structure theoretically (Jordan-H\"older theory, Sylow theory), and as geometries 
(distance-regular graphs, designs).  Similarly, 
association schemoids may be studied relying on combinatorial way, 
categorical representation theory and homotopy theory for small categories \cite{Thomason, Latch}. 
In fact, 
the diagram (6.1) of categories and functors enables us to 
expect that schemoids bring us considerable interests containing association schemes 
and that the study of the new subjects paves the way for homotopical and 
categorical consideration of such generalized groups. 

As one of further investigations on schemoids, we intend to discuss a (co)fibration category structure  \cite{B} 
on an appropriate category of schemoids. In particular, 
it is important to consider (co)cylinder objects explicitly in the category in developing 
a homotopical classification of schemoids.  
Moreover, the notion of (co)limits in a category may give us a new construction of schemoids and hence association 
schemes. These are issues in our forthcoming paper. 


\section{Association schemoids}
We begin by recalling the definition of the association scheme. 
Let $X$ be a finite set and $S$ a partition of $X\times X$, namely a subset of the power set $2^{X\times X}$, 
which contains a partition of the subset $1_X :=\{ (x, x) \mid x \in X\}$ as a subset. 
Assume further that for each 
$g \in S$, the subset $g^*:=\{(y, x) \mid (x, y) \in g\}$ is in $S$. Then the pair $(X, S)$ is called a 
{\it coherent configuration} if for all 
$e, f, g \in S$, there exists an integer $p_{ef}^g$ such that for any $(x, z) \in g$ 
$$
p_{ef}^g=\sharp \{y \in X \mid (x, y)\in e \ \text{and} \ (y, z) \in f \}. 
$$
Observe that $p_{ef}^g$ is independent of the choice of $(x, z) \in g$. 
By definition, a coherent configuration $(X, S)$ is  an {\it association scheme} 
if $S$ contains the subset $1_X$ as an element.  

Let $K$ be a group acting a finite set $X$. Then $K$ act on the set $X\times X$ diagonally. 
We see that the set $S_K$ of $G$-orbits of $X\times X$ gives rise to a coherent configuration $(X, S_K)$.  
It is readily seen that $(X, S_K)$ is an association scheme if and only if the action of $K$ on $X$ is transitive. 

For an association scheme $(X, S)$, the pair $(x, y) \in X\times X$ is regarded as 
an edge between vertices $x$ and $y$. Then 
the scheme $(X, S)$ is considered as a directed complete graph 
and hence a small category; see Example \ref{ex:ex1} (ii) below for more details. 
With this in mind, 
we generalize the notion of association schemes from a categorical point of view. 

\begin{defn}\label{defn:schemoid}
Let $\C$ be a small category; that is, the class of the objects of the category $\C$ is a set. 
Let $S:=\{\sigma_l\}_{l\in I}$ be a partition of the set $mor(\C)$
of all morphisms in $\C$. We call the pair $(\C, S)$ a {\it quasi-schemoid} if 
the set $S$ satisfies the {\it concatenation axiom}. This means that for a triple $\sigma, \tau, \mu \in S$ 
and for any morphisms $f$, $g$ in $\mu$, as a set 
$$
(\pi_{\sigma\tau}^\mu)^{-1}(f) \cong (\pi_{\sigma\tau}^\mu)^{-1}(g), 
$$ 
where $\pi_{\sigma\tau}^\mu : \pi_{\sigma\tau}^{-1}(\mu) \to \mu$ is the map defined to be the restriction of the 
concatenation map $\pi_{\sigma\tau} : \sigma \times_{ob(\C)}\tau \to mor(\C)$.
\end{defn}

For $\sigma, \tau$ and $\mu \in S$, 
we have a diagram which explains the condition above 
$$
\xymatrix@C35pt@R20pt{
\left( \pi_{\sigma\tau}\right) ^{-1}(\mu) \ar@{^{(}->}[rr] \ar[d]_{\pi_{\sigma\tau}^\mu} && \sigma\times _{ob (\C)}\tau \ar[r] \ar[d] \ar[ld]_{comp=\pi_{\sigma\tau}} & \tau
 \ar[d]^{t}\\
\mu \ar@{^{(}->}[r] & mor (\C) & \sigma\ar[r]_{s} & ob (\C). 
}
\eqnlabel{add-0}
$$
If the set $(\pi_{\sigma\tau}^\mu)^{-1}(f)$ is finite, 
then we speak of the number $p_{\sigma\tau}^\mu:=\sharp (\pi_{\sigma\tau}^\mu)^{-1}(f)$ 
as the {\it structure constant}.  

\begin{defn}\label{defn:association_schemoids}
A quasi-schemoid  $(\C, S)$ is  an {\it association schemoid} (schemoid for short)  
if the following conditions (i) and (ii) hold. 

(i) For any $\sigma \in S$ and the set $J:= \amalg_{x\in ob(\C)}\text{Hom}_\C(x, x)$, if $\sigma \cap J \neq \phi$, then 
$\sigma \subset J$.  

(ii) There exists a contravariant functor 
$T : \C \to \C$ such that $T^2=id_\C$ and 
$$
\sigma^*:=\{ T(f) \mid f \in \sigma \}
$$
is in the set $S$ for any $\sigma \in S$. We denote by $(\C, S, T)$ the association schemoid together with such a functor $T$.
\end{defn}

Let $J_0$ denote the subset $\{ 1_x \mid x\in ob (\C)\}$ of the set of morphisms of a category $\C$. 
We call a (quasi-)schemoid {\it unital} if  $\alpha \subset J_0$ 
for any $\alpha\in S$ with $\alpha \cap J_0\neq\phi$.

We define morphisms between (quasi-)schemoids. 

\begin{defn}\label{defn:morphisms}
(i) Let $(\C, S)$ and $(\E, S')$ be quasi-schemoids. A functor $F: \C \to \E$ is a morphism of quasi-schemoids 
if for any $\sigma$ in $S$, 
$F(\sigma) \subset \tau$ for some $\tau$ in $S'$. 
We then write $F : (\C, S) \to (\E, S')$ for the morphism.  By abuse of notation, we may write 
$F(\sigma) = \tau$ if $F(\sigma) \subset \tau$ for a morphism $F$ of schemoids. 

(ii) Let $(\C, S, T)$ and $(\E, S', T')$ be association schemoids. 
If a morphism $F$ from  $(\C, S)$  to $(\E, S')$ satisfies the condition that 
$FT = T'F$, then we call such a functor $F$ a morphism of association schemoids and denote it by 
$F: (\C, S, T) \to (\E, S',  T')$. 
\end{defn}

Let $\C$ be a small category and $\K$ a commutative ring with unit.  
We here recall the {\it category algebra} $\K\C$ of $\C$ which is defined to be the free $\K$-module generated 
by the morphisms of the category $\C$. The product of morphisms $\alpha$ and $\beta$ as elements of $\K\C$ is defined by 
$$
\alpha\beta = \begin{cases}
\alpha \circ \beta  & \text{if $\alpha$ and $\beta$ are composable}  \\
0 & \text{otherwise}. 
\end{cases}
$$

Let $(\C, S)$ be a quasi-schemoid with $mor (\C)$ finite.  
Then for any $\sigma$ and $\tau$ in $S$, we have an equality 
$$
(\sum_{s\in \sigma} s) \cdot (\sum_{t\in \tau} t) = \sum_{\mu\in S} p_{\sigma\tau}^\mu (\sum_{u\in \mu} u)
$$ 
in the category algebra $\K \C$ of $\C$. This enables one to obtain 
a subalgebra 
$
\K(\C, S)
$
of $\K\C$ generated by the elements $(\sum_{s\in \sigma} s)$ for all $\sigma \in S$. We refer to  
the subalgebra 
$
\K(\C, S)
$
as the {\it schemoid algebra} of $(\C, S)$. 
Observe that the algebra $\K(\C, S)$ is not unital in general even if $\C$ is finite. The following lemma shows the significance of the unitality of a (quasi-)schemoid. 

\begin{lem}\label{lem:unital}
Let $(\C, S)$ be a quasi-schemoid whose underlying category $\C$ is finite. Then 
$(C, S)$ is unital if and only if so is the schemoid algebra $\K(\C, S)$. 
\end{lem}

\begin{proof}
Assume that $\K(\C, S)$ is unital. We write $\sum_{x\in ob(\C)}1_x = \sum_i\alpha_i(\sum_{s\in \sigma_i}s)$, where 
$\alpha_i \in \K$ and $\sigma_i \in S$. Then for any $x \in ob(\C)$, there exists a unique index $i$ such that 
$1_{x} \in \sigma_i$ and $\alpha_i = 1$. If the element $\sigma_i$ of $S$ contains a morphism $s$ which is not the identity $1_y$ 
for some $y \in ob(\C)$, then the right hand side of the equality has $s$ as a term, which is a contradiction. 
The converse is immediate.       
\end{proof}

We are aware that the schemoid algebra is a generalization of the Bose-Mesner algebra associated 
with an association scheme; see Example \ref{ex:ex1} (ii) below for details.  
We may call the schemoid algebra the Bose-Mesner algebra of the given quasi-schemoid. 

Suppose that the underlying category $\C$ of a quasi-schemoid $(\C, S)$ has a terminal object $e$. 
By definition, for any object $x$ of $\C$, 
there exists exactly only one morphism $(e, x)$ from $x$ to $e$. For any $\sigma \in S$, we define an element 
$E_\sigma$ of the category algebra $\K\C$ by 
$E_\sigma = \sum_{(e, x)\in \sigma}1_x$. We refer to the subalgebra $T(e)$ of $\K\C$ 
generated by $\K(\C, S)$ and elements $E_\sigma$ for $\sigma \in S$ as the Terwilliger algebra 
of $(\C, S)$. 
Since $\sum_{\sigma \in S}E_\sigma = \sum_{x \in ob(\C)}1_x$, it follows that $T(e)$ is unital if $\C$ is finite.

\begin{rem}
(i) The schemoid algebra  of  an quasi-schemoid $(\C, S)$ can be defined provided 
$\sharp \sigma < \infty$ for each $\sigma \in S$ and for any $\tau$ and $\mu$ in $S$, the structure constant 
$p_{\sigma\tau}^\mu$ is zero except for 
at most finite indexes $\mu \in S$. 

(ii) A functor $F :  \C \to \E$ induces an algebra map $F : \K\C \to \K\E$ if $F$ is a monomorphism on objects.  
However, a morphism $F : (\C, S) \to (\E, H)$ of quasi-association schemoids does not define naturally an algebra map between 
schemoid algebras $\K(\C, S)$ and $\K(\C, E)$ even if $F$ induces an algebra map as mentioned above.  
In Section \ref{section:6}, we shall discuss morphisms between quasi-schemoids which induce algebra maps 
between the schemoid algebras. 
\end{rem}

\begin{ex} \label{ex:ex1} (i) A (possibly infinite) group $G$ gives rise to an association schemoid $(G, \widetilde{G}, T)$, where
$\widetilde{G} = \{ \{g\} \mid g \in G\}$ and $T(g)=g^{-1}$. The schemoid algebra $\K(G, \widetilde{G})$ is nothing but 
the group ring $\K G$. 

(ii) For an association scheme $(X, S)$, we define an association schemoid $j(X, S)$ by the triple $(\C, U, T)$ for which 
$ob(\C)=X$, $\text{Hom}_\C(y, x) =\{(x, y)\} \subset X\times X$, $U =S$, $T(x) = x$ and $T(x, y) = (y, x)$, where the composite of morphisms 
$(z, x)$ and $(x, y)$ is defined by $(z, x) \circ (x, y) = (z, y)$.    

The schemoid algebra of $j(X, S)$ is indeed the ordinary Bose-Mesner algebra of the association scheme $(X, S)$. 
Moreover, we see that the Terwilliger algebra $T(e)$ of $j(X, S)$ is 
the Terwilliger algebra of $(X, S)$ introduced originally in \cite{T}.  
Observe that every object of $j(X, S)$ is a terminal one because $j(X, S)$ is a directed complete graph.  

(iii) Let $G$ be a group. Define a subset $G_f$ of $G\times G$ for $f\in G$ by 
$G_f:=\{(k, l) \mid k^{-1}l = f\}$. Then we have an association scheme $S(G) = (G, [G], T)$, where 
$[G]=\{G_f\}_{f\in G}$. 
The same procedure permits us to obtain 
an association schemoid $\widetilde{S}({\mathcal H}) = (\widetilde{\mathcal H}, S, T)$ for a groupoid 
${\mathcal H}$, where $ob (\widetilde{\mathcal H}) = mor({\mathcal H})$ and 

$$
\text{Hom}_{\widetilde{{\mathcal H}}}(g, h) = \begin{cases}
\{(h, g)\}  & \text{if}  \ \  t(h) = t(g)  \\
\varnothing & \text{otherwise}. 
\end{cases}
$$
In fact, we define the partition $S=\{ \G_f \}_{f \in mor({\mathcal H})}$ by $\G_f = \{(k, l) \mid k^{-1}l = f\}$, 
$T(f)= f$ for $f$ in $ob(\widetilde{\mathcal H})$ 
and $T((f, g)) = (g, f)$ for $(f, g) \in mor(\widetilde{\mathcal H})$. 
\end{ex}

We define categories $q{\bf ASmd}$ and ${\bf ASmd}$ to be the category of quasi-schemoids and that of association schemoids, respectively.   
The forgetful functor $k : {\bf ASmd} \to q{\bf ASmd}$ is defined immediately. 

Let ${\bf Cat}$ be the category of small categories. We recall that morphisms of ${\bf Cat}$ 
are functors between small categories. 
Let ${\bf Gpd }$ be the category of possibly infinite groupoids which is a full subcategory of ${\bf Cat}$. 
For a functor $F \in \text{Hom}_{\bf Gpd }(\calK, \calH)$, we define a morphism  
$\widetilde{S}(F)$ in $\text{Hom}_{{\bf ASmd}}(\widetilde{S}(\calK), \widetilde{S}(\calH))$ by 
$\widetilde{S}(F)(f) = F(f)$ and $\widetilde{S}(F)(f, g) = (F(f), F(g))$ for $f, g \in mor(\calK)$. 
Then the correspondence gives rise to a functor $\widetilde{S}( \ ) : {\bf Gpd } \to {\bf ASmd}$. 

Let ${\bf AS}$ be the category of association schemes in the sense of Hanaki \cite{H}; that is, 
its objects are association schemes and morphisms $f : (X, S) \to (X', S')$ are maps which satisfy the condition that 
for any $s\in S$, $f(s) \subset s'$ for some $s' \in S'$.   
It is readily seen that the correspondence $j$ defined in Example \ref{ex:ex1} (ii) induces 
a functor $j : {\bf AS} \to {\bf ASmd}$.

We obtain many association schemoids from association schemes and groupoids via the functors $\widetilde{S}$ and $j$; see 
Example \ref{ex5}.
As mentioned above, we have 

\begin{lem}\label{lem:images}
A schemoid in the image of the functor $\widetilde{S}$ or $j$ is a groupoid whose hom-set for any two objects consists of 
a single element. 
\end{lem}

The following examples are association (quasi-)schemoids which are in neither of the images.  A more systematic way to 
construct (quasi-)schemoids is described in Sections 5 and 7. 

\begin{ex}\label{ex_group} We consider a group $G$ a groupoid with single object. 
Then the triple $G^\bullet:= (G, \{G\}, T)$ is a schemoid with a contravariant functor 
$T :  G \to G$ defined by $T(g) = g^{-1}$. 
In view of Lemma \ref{lem:images}, we see that the schemoid $G^\bullet$ is in neither the image 
of the functors $j$ nor the image of $\widetilde{S}( \ )$ if $\sharp G >  1$.  
\end{ex}

\begin{ex}\label{ex1} Let us consider a category $\C$ defined by the diagram
$$
\xymatrix{
x \ar@(lu,ld)_{1_x} \ar[r]^{f} & y \ar@(ru,rd)^{1_y}
}
$$
Define a contravariant functor $T$ on $\C$ by $T(x) = y$ and $T(y) = x$. Then the triple $(\C, S, T)$ is a unital schemoid, where  
$S=\{ S_1, S_2\}$ with $S_1=\{ 1_x, 1_y\}$ and $S_2=\{ f\}$. We can define another partition $S'$ by 
$S'=\{ S'_1, S'_2, S'_3\}$ for which  $S'_1=\{ 1_x\}$, $S'_2=\{1_y\}$ and $S'_3=\{ f\}$. Then $(\C, S', T)$ is also a unital schemoid. 
\end{ex}

\begin{ex}\label{ex:Join}
Let $\C$ and $\D$ be categories. The join construction $\C\ast\D$ with $\C$ and $\D$ is a category given as follows. The set of objects is the disjoint union 
$ob(\C)\cup ob(\D)$. 
The set of morphisms consists of all elements of $mor(\C) \cup mor (\D)$ and  $w_{ab} \in \text{Hom}_{\C\ast\D}(a, b)$ 
for $a \in ob(\C)$ and $b \in ob(\D)$. 
Observe that $\text{Hom}_{\C\ast\D}(a, b)$ has exactly one element $w_{ab}$ and 
$\text{Hom}_{\C\ast\D}(b, a) = \phi$ if  $a \in ob(\C)$ and $b \in ob(\D)$. The additional concatenation law is defined by 
$\alpha w_{as} = w_{at}$ and $w_{vb} \beta = w_{ub}$ for $\alpha \in \text{Hom}_\D(s, t)$ and $\beta \in \text{Hom}_\C(u, v)$. 

Let $(\C, S)$ and $(\D, S')$ be quasi-schemoids. We define a partition $\Sigma$ of $mor(\C\ast\D)$ by 
$
\Sigma = S\cup S' \cup \{ \{w_{ab}\} \}_{a \in ob(\C), b \in ob(\D)}.
$ 
It is readily seen that $(\C\ast\D, \Sigma)$ is a quasi-schemoid.  
\end{ex}

\begin{ex}\label{ex2}
Let $G$ be a group and let $\C$ denote the category $G\ast G^{op}$ obtained by the join construction, namely a category with   
$ob(\C)=\{x, y\}$, $\mathrm{Hom}_{\C}(x, x) =G$, $\mathrm{Hom}_{\C}(y, y) =G^{\mathrm{op}}$, 
$\mathrm{Hom}_{\C}(x, y) =\{ f\}$ and $\mathrm{Hom}_{\C}( y, x) =\phi$. The diagram 
$$
\xymatrix{
x \ar@(lu,ld)_{G} \ar[r]^{f} & y \ar@(ru,rd)^{G^{\mathrm{op}}}
}
$$
denotes the category $\C$. It is shown that $T : \C \to \C$ defined by $T(x) = y$ and $T(y) = x$ is a contravariant functor. 
Then we have a unital schemoid  $(\C, S, T)$ with the partition $S$ defined by 
$S=\{ S_g\} _{g\in G}\cup\{ S_f\}$, where $S_g=\{g, g^{\mathrm{op}}\}$ and $S_f=\{ f\}$.

Observe that $(\C, S)$ is not isomorphic to the join $(G^\bullet\ast {G'^\bullet}, \Sigma)$ of the schemoid 
$G^\bullet$ and its copy ${G'^\bullet}$ in the sense of Exmaple \ref{ex:Join}. In fact, for any morphism 
$F : (G^\bullet\ast {G'^\bullet}, \Sigma)\to (\C, S)$ of quasi-schemoids, we see that $F(1_G) = 1_x$ and $F(1_{G'}) = 1_y$. This implies that 
$F(\{G\}\cup\{G'\}) \subset S_e$. 
\end{ex}

\begin{ex}\label{ex3}
Let $\C$ be a category defined by the diagram 
$$
\xymatrix@C30pt@R15pt{
\ar@{.}[r] & \ar[r] & y_{i-1} \ar@(lu,ru)^{1_{y_{i-1}}} && x_i \ar@(lu,ru)^{1_{x_i}} \ar[ll]_{h_{i-1}} \ar[d]_{f_i} \ar[rr]^{g_i} && y_{i+1} \ar@(lu,ru)^{1_{y_{i+1}}} & \ar[l] \ar@{.}[r] &\\
\ar@{.}[r] & & x_{i-1} \ar[l] \ar@(ld,rd)_{1_{x_{i-1}}} \ar[u]^{f_{i-1}} \ar[rr]^{g_{i-1}} && y_i \ar@(ld,rd)_{1_{y_i}} && x_{i+1} \ar@(ld,rd)_{1_{x_{i+1}}} \ar[ll]_{h_i} \ar[u]^{f_{i+1}} \ar[r] & \ar@{.}[r] &
}
$$
We define subsets $\sigma$ and $J_0$ of $mor (\C)$ by 
$\sigma=\{g_i, h_i\}_{i \in \mathbb{Z}}$ and $J_0 =\{1_{x_i}, 1_{y_i}\}_{i \in \mathbb{Z}}$, respectively. 
Then for any partition $S$ of the form $\{ \sigma, J_0, \tau_l\}_{l \in I}$ of  $mor (\C)$, 
the triple $(\C, S, T)$ is a unital schemoid, where $T(x_i)=y_i$ and $T(y_i)=x_i$  for any $i \in \mathbb{Z}$. 
\end{ex}

\begin{ex}\label{ex4} For $l \geq 1$, 
let $C_l$ be a category defined by the diagram 
$$
\xymatrix@C35pt@R10pt{
 & a_l \ar[rd]^{\beta_l} & \\ 
x \ar[rr]^{\varepsilon} \ar[ru]^{\alpha_l} \ar[rd]_{\gamma_l} & & y & \text{with} \ \  \beta_l \alpha_l = \varepsilon = \delta_l\gamma_l;  \\
 & b_{l} \ar[ru]_{\delta_l} & 
}
$$
\text{see} \cite[(7.8)]{B-W}. We define a partition $S=\{S_l^i\}_{i=0, 1, 2, 3}$ of $mor(\C_l)$ by 
$S_l^1=\{\alpha_l, \gamma_l\}$, 
$S_l^2 =\{\beta_l, \delta_l\}$, $S_l^3=\{\varepsilon\}$ and $S_l^0=\{1_x, 1_y, 1_{a_l}, 1_{b_l}\}$. 
Define a contravariant functor $T : C_l \to C_i$ by $T(a_l) = b_l$, $T(\varepsilon) = \varepsilon$, 
$T(\alpha_l)=\delta_l$ and $T(\beta_l) = \gamma_l$. Then we obtain a unital schemoid  $C_{[k]}$ of the form 
$$( \bigcup_{1\leq l \leq k}\C_l,  \{\bigcup_{1\leq l \leq k}S_l^i\}_{0\leq i\leq 3}, T). $$ 
\end{ex}

\begin{ex}\label{ex5}
Let $(\C_i, S_i, T_i)$ be a schemoid. Then it is readily seen that the product $(\Pi_i \C_i, \Pi_i S_i, \Pi_iT_i)$ is a schemoid.  In particular 
an EI-category of the form $C_{[k]}\times G^\bullet$; that is, all endomorphisms are isomorphisms, 
is a schemoid for any group $G$. Moreover, for an association scheme $(X, S)$, 
we have a schemoid of the form $j(X, s)\times G^\bullet$, which is in neither the images of 
$j$ nor the image of $\widetilde{S}( \ )$ provided 
$\sharp G >  1$.  
\end{ex}

\section{A category of association schemoids and related categories}

Let ${\bf Gr}$ be category of finite groups.   
With the funcotrs $\widetilde{S}( \ )$, $j$ and $k$ mentioned in Section 2, 
we obtain a diagram of categories and functors
$$
\xymatrix@C35pt@R18pt{
{\bf Gpd } \ar[r]^-{{\widetilde S}( \ )} \ar@/^2.5pc/[rrr]_-{\ell} & {\bf ASmd} \ar[r]^k & 
q{\bf ASmd} \ar@<1ex>[r]^-{U} 
& {\bf Cat}. \ar@<1ex>[l]^-{K}  \\
{\bf Gr} \ar[u]^i \ar[r]^-{S( \ )}  & {\bf AS} \ar[u]_{j} 
}
\eqnlabel{add-1}
$$ 
Here $U$ is the forgetful functor and,  
for a small category $\C$, 
$K$ defines a quasi-schemoid 
$K(\C)=(\C, S)$ with $S = \{ \{f\} \}_{f \in mor (\C)}$.  
It is readily seen that $K$ is a fully faithful functor and that $UK = id_{\bf Cat}$. 
Observe that 
$U\circ k \circ \widetilde{S}( \ )$ does {\it not} coincide with the canonical faithful functor  
$\ell : {\bf Gpd } \to {\bf Cat}$. We emphasize that the left-hand square is commutative. 

\begin{rem}\label{rem:categories}
(i) The functor $j$ factors through the category of coherent configurations, whose morphisms are defined by the same way as in ${\bf AS}$. \\
(ii) We see that the functor $K$ is the left adjoint of the forgetful functor $U$ and that the schemoid algebra of $K(\C)$ is the whole category algebra $\K(\C)$.  
\end{rem}

\begin{thm}\label{thm:functors}
{\em (i)} The functors $i$ and $j$ are fully faithful embedding. \\
{\em (ii)} The functors $S( \ )$ and $\widetilde{S}( \ )$ are faithful.    
\end{thm}

\begin{proof} We prove the assertion (i).  It is well known that $i$ is fully faithful. 
Let $(X,S_X)$ and $(Y,S_Y)$ be association schemes. 
It is readily seen that $i$ and $j$ are injective on the sets of objects. We prove that 
$$ j : \text{Hom}_{\bf AS}((X, S_X) , (Y, S_Y)) \to \text{Hom}_{{\bf ASmd}}( j(X, S_X) , j(Y, S_Y))$$ is bijective. 
Let $F$ be a morphism from $j(X,S_X)=(\C_X, U_X, T_X)$ to $j(Y,S_Y)=(\C_Y, U_Y, T_Y)$. 
Now we define a map $\varphi (F):X \to Y$ by $\varphi (F)(x)=F(x)$, where $x\in X=ob \bigl(j(X,S_X)\bigr)$. 
For each $s \in U_X$, there exists a unique set $t$ of morphisms in $U_Y$ such that $F(s)\subset t$ and $F(s^*)\subset t^*$. 
The map $\varphi(F):S_X\to  S_Y$ defined by $\varphi(F)(s)=t$ fits into the commutative diagram
$$
\xymatrix@C40pt@R18pt{
X\times X \ar[r]^{\varphi(F)\times \varphi (F)} \ar[d]_{r} & Y\times Y \ar[d]^{r}\\
S_X \ar[r]^{\varphi (F)} & S_Y,
}
$$
where $r$ is the map defined by $r(x_1,x_2)=s$ for $(x_1,x_2)\in s$. 
Moreover, we see that $\varphi$ is the inverse of $j$. This completes the proof. 

\noindent
(ii) We prove that the map 
$$
\widetilde{S}:=k\circ \widetilde{S}( \ ) : \text{Hom}_{\bf Gpd }(\calK, \calH) \to \text{Hom}_{q{\bf ASmd}}(\widetilde{S}(\calK), \widetilde{S}(\calH))
$$
is injective. To this end, a left inverse of $\widetilde{S}$ is constructed below. 

Let $G: \widetilde{S}(\calK) \to\widetilde{S}(\calH)$ be a morphism in $q{\bf ASmd}$, namely a functor which gives maps 
$G : mor (\calK) \to mor (\calH)$ and $G : \text{Hom}_{\widetilde{S}(\calK)}(f, g) \to \text{Hom}_{\widetilde{S}(\calH)}(G(f), G(g))$.  
The hom-set of $\widetilde{S}(\calH)$ consists of a single element. Then we see that $G(f, g) =(G(f), G(g))$. 

\begin{claim}\label{claim:1_1}
For an object $f \in ob \widetilde{S}(\calK) = mor (\calK)$, one has 
$sG(f) = sG(1_{s(f)})$ and $tG(1_{t(f)}) = tG(f)$. 
\end{claim}

\begin{claim}\label{claim:2}
For composable morphisms $f : s(f) \to t(f)$ and $g : s(g) = t(f) \to t(g)$ in $\calK$, 
$
G(fg) = G(f)G(1_{t(g)})^{-1}G(g).
$
\end{claim}

Define a map $\overline{( \ )} : \text{Hom}_{q{\bf ASmd}}(\widetilde{S}(\calK), \widetilde{S}(\calH)) \to \text{Hom}_{\bf Gpd }(\calK, \calH)$
by $\overline{(G)}(x)=sG(1_x)$  for $x \in ob(\calK)$ and $\overline{(G)}(f)=G(1_{t(f)})^{-1}G(f)$ for $f \in \text{Hom}_{\calK}(x, y)$. Claim \ref{claim:1} 
implies that the composite 
$$
\xymatrix@C8pt@R15pt{
\overline{(G)}(f) : \overline{(G)}(s(f))=sG(1_{s(f)}) = sG(f) \ar[rrr]^-{G(f)}  & &  & tG(f) \ar@{=}[r] & tG(1_{t(f)}) \ar[d]^-{G(1_{t(f)})^{-1}} \\
&   & &  & sG(1_{t(f)}) \ar@{=}[r] & \overline{(G)}(t(f))
}
$$
is well-defined. We then have $\overline{(G)}(1_x)= G(1_{t(1_x)})^{-1}G(1_x) = 1_{sG(1_x)} = 1_{\overline{(G)}(x)}$. Moreover, 
Claim \ref{claim:2} enables us to deduce that 
$$
\overline{(G)}(f)\overline{(G)}(g) = G(1_{t(f)})^{-1}G(f)G(1_{t(g)})^{-1}G(g)= G(1_{t(fg)})G(fg) = \overline{(G)}(fg).
$$
Thus $\overline{(G)} : \calK \to \calH$ is a functor for any $G$ in $\text{Hom}_{q{\bf ASmd}}(\widetilde{S}(\calK), \widetilde{S}(\calH))$ so that the map $\overline{( \ )}$ is well-defined. 
It is readily seen that the composite $\overline{( \ )}\circ \widetilde{S}$ is the identity map. 

Since the left-hand side in the diagram (3.1) is commutative and $\widetilde{S}( \ )\circ i$ is faithful, it follows that so is $S(\ )$. 
This completes the proof. 
\end{proof}

\begin{proof}[Proof of Claim \ref{claim:1_1}] We can write $(\widetilde{\calK}, \{\calK_f\}_{f\in mor(\calK)})$ and 
$(\widetilde{\calH}, \{\calH_g\}_{g\in mor(\calH)})$ for $\widetilde{S}(\calK)$ and $\widetilde{S}(\calH)$, respectively; see Example \ref{ex:ex1} (iii). 
Since $(f, f)$ and $(1_{s(f)}, 1_{s(f)})$ are in $\calK_{1_{s(f)}}$, it follows that $(G(f), G(f))$ and $(G(1_{s(f)}), G(1_{s(f)}))$ are in the same 
$\calH_l$ for some $l \in mor \calH$. This yields that 
$G(f)^{-1}G(f) = l  = G(1_{s(f)})^{-1}G(1_{s(f)})$ and hence $sG(f) = sG(1_{s(f)})$. 
We have $tG(1_{t(f)}) = tG(f)$ as $G(1_{t(f)}, f) = (G(1_{t(f)}), G(f))$. 
\end{proof}

\begin{proof}[Proof of Claim \ref{claim:2}]
We observe that $(1_{t(g)}, g)$ and $(f, fg)$ are in $\calK_g$. Then morphisms $(G(1_{t(g)}), G(g))$ and $(G(f), G(fg))$ are in 
$\calH_h$ for some $h \in mor ({\calH})$. 
This implies that $G(1_{t(g)})^{-1}G(g) = h = G(f)^{-1}G(fg)$. We have the result. 
\end{proof}

Let $(q{\bf ASmd})_0$ be the category of quasi-schemoids with base points; that is, an object $(\C, S)$ 
in $(q{\bf ASmd})_0$ is a quasi-schemoid with $\C^\circ$ a subset of $ob(\C)$ 
and a morphism $F : (\C, S) \to (\E, T)$ preserves the sets of base points in the sense that 
$F(\C^\circ) \subset \E^\circ$. For a groupoid ${\mathcal G}$, the quas-schemoid $\widetilde{S}({\mathcal G})=(\widetilde{\mathcal G}, S)$ 
is endowed with base points $\widetilde{\mathcal G}^\circ= \{1_x\}_{x\in ob({\mathcal G})}$. 
We define the category of schemoids  $({\bf ASmd})_0$ with base points as well.

\begin{cor}\label{cor:fully_faithful}
The functor $\widetilde{S} : {\bf Gpd } \to (q{\bf ASmd})_0$ is fully faithful. 
\end{cor}

\begin{proof}
We define a map 
$\overline{( \ )} : \text{Hom}_{(q{\bf ASmd})_0}(\widetilde{S}(\calK), \widetilde{S}(\calH)) \to \text{Hom}_{\bf Gpd }(\calK, \calH)$ 
by the same functor as $\overline{( \ )}$ in the proof of Theorem \ref{thm:functors} (ii). 
Since $G(1_x)$ is the identity map for a morphism $G : \widetilde{S}(\calK) \to \widetilde{S}(\calH)$  in 
$(q{\bf ASmd})_0$, it follows that $\overline{(G)}(f)=G(f)$ for any $f \in \text{Hom}_{\calK}(x, y)$. It turns out that the map $\overline{( \ )}$ is the inverse of  $\widetilde{S}$. 
\end{proof}

\begin{rem}\label{rem:quasi-schemoids}
We have a commutative diagram
$$
\xymatrix@C35pt@R15pt{
& \text{Hom}_{({\bf ASmd})_0}(\widetilde{S}(\calK), \widetilde{S}(\calH)) \ar[d]^{U} \\ 
\text{Hom}_{\bf Gpd }(\calK, \calH) \ar[ur]^{\widetilde{S}( \ )} \ar[r]_-{\widetilde{S}( \ )}^-{\approx} & 
\text{Hom}_{(q{\bf ASmd})_0}(\widetilde{S}(\calK), \widetilde{S}(\calH)), 
}
$$
where $U$ denotes the map induced by the forgetful functor. For any functor $G$ in 
$\text{Hom}_{(q{\bf ASmd})_0}(\widetilde{S}(\calK), \widetilde{S}(\calH))$ and for a morphism $(f, g)$ in 
$\widetilde{S}(\calK)$, it follows that 
$$GT((f, g)) = G((g, f))=(G(g), G(f))=TG((f, g))$$ and hence $G$ is also in  
$\text{Hom}_{({\bf ASmd})_0}(\widetilde{S}(\calK), \widetilde{S}(\calH))$. This yields that the vertical arrow $U$ is a bijection. 
\end{rem}

\section{Thin association schemoids}\label{section:thin association shcemoids}

The goal of this section is to prove that the category of groupoids is equivalent to the category of 
based {\it thin} association schemoids, which is a subcategory of ${\bf ASmd}$. 

A thin association schemoid defined below is a generalization of a thin coherent configuration in the sense of 
Hanaki and Yoshikawa \cite{H-Y}. The results \cite[Theorem 12, Remark 16]{H-Y} assert that a connected finite groupoid is essentially identical with 
a finite thin coherent configuration.  We consider such a correspondence from a categorical point of view. 

Let  $(\C, S, T)$ be an association schemoid. 
For $\sigma$, $\tau$ and  $\mu\in S$, we recall the structure constant 
$p^{\mu}_{\tau\sigma} = \sharp\left(\pi_{\tau\sigma}^\mu\right) ^{-1}(f)$, where $f\in\mu$; see Definition \ref{defn:schemoid}. 

\begin{defn} \label{defn:semi-thin} (Compare the definition of a thin coherent configuration \cite[Section 3]{H-Y} )
A unital association schemoid $(\mathcal{C}, S, T)$ is called \textit{semi-thin} if the following two conditions hold. 

(i) $\sharp\{ f\in\sigma \mid s(f)=x\}\leq 1$ for any $\sigma\in S$ and $x\in ob (\C)$. 

(ii) The underlying category $\mathcal{C}$ is a groupoid with the contravariant functor $T : \C \to \C$ defined by 
$T(f)=f^{-1}$ for $f\in mor(\C)$.
\end{defn}


Following Zieschang \cite{Z_1995, Z_book} and Hanaki and Yoshikawa \cite{H-Y},  we here fix the notation used below. 
We define subsets $S_J$ and $S_0$ of $S$ by 
$S_J=\{\kappa\in S \mid \kappa\cap J\not=\phi\}$ and $S_0=\{\alpha\in S \mid \alpha\cap J_0\not=\phi\}$, respectively. 
For any $\sigma\in S$, write $\sigma x=\{f\in\sigma \mid s(f)=x\}$ and $y \sigma =\{f\in\sigma \mid t(f)=y\}$, where 
 $x, y\in ob (\C)$. 
For any $\alpha\in S_0$, we write $X_{\alpha}=\{x\in ob (\C) \mid 1_x\in\alpha\}$. 
Let  
${_{\alpha}S_{\beta}}$ be the subset of $S$ defined by 
$${_{\alpha}S_{\beta}}=\{\sigma\in S \mid p^{\sigma}_{\sigma\alpha}=p^{\sigma}_{\beta\sigma}=1\},$$ where 
$\alpha , \beta\in S_0$. 

To construct a functor from the category of semi-thin association schemoids to the category of groupoids, 
we need some lemmas.
\begin{lem}\label{Lemma1} {\em (}cf. \cite[Lemma 1]{H-Y}{\em )}
Let $(\mathcal{C}, S, T)$ be a unital association schemoid. 

{\em (i)} For any $\sigma\in S$, there exists a unique element $\alpha$ in $S_0$ such that $p^{\sigma}_{\sigma\alpha}=1$. 
Moreover, $p^{\sigma}_{\sigma\alpha '}=0$ if $\alpha '\in S_0$ and $\alpha '\not=\alpha$.

{\em (ii)} For any $\sigma\in S$, there exists a unique element $\beta$ in $S_0$ such that $p^{\sigma}_{\beta\sigma}=1$. 
Moreover,  $p^{\sigma}_{\beta '\sigma}
=0$ if $\beta '\in S_0$ and $\beta '\not=\beta$.
\end{lem}
Lemma \ref{Lemma1} allows one to deduce that  
$$
S=\coprod _{\alpha, \beta\in S_0}{_{\alpha}S_{\beta}}.
$$
\begin{proof}[Proof of Lemma \ref{Lemma1}] We prove (i). The second assertion follows from the same argument as in the proof of (i). 
Let $f$ be a morphism in $\sigma$. Suppose that $s(f)\not\in X_{\alpha}$. 
If $p^{\sigma}_{\sigma\alpha}\geq 1$, then there exists $g\in\sigma$ such that $s(g)\in X_{\alpha}$ and $g\circ 1_{s(g)}=f$. Since $g=g\circ 1_{s(g)}=f$, 
we see that $s(f)=s(g)\in X_{\alpha}$, which is a contradiction. This yields that $p^{\sigma}_{\sigma\alpha}=0$ if $s(f) \notin X_\alpha$. 

It is readily seen that $ob (\C) = \coprod_{\alpha \in S_0}X_\alpha$. 
Then there exists a unique element $\alpha \in S_0$ such that $s(f)\in X_{\alpha}$.  
This allows us to deduce  that $\left( f, 1_{s(f)}\right)\in\bigl( \pi_{\sigma\alpha}\bigr) ^{-1}(f)$ and hence 
$p^{\sigma}_{\sigma\alpha}\geq 1$. On the other hand, if $\left( g, 1_{s(g)}\right)\in\bigl(\pi_{\sigma\alpha}\bigr) ^{-1}(f)$, then $g=g\circ 1_{s(g)}=f$. Therefore we have $p^{\sigma}_{\sigma\alpha}=1$.
\end{proof}

\begin{lem}\label{Lemma2}
Let $(\mathcal{C}, S, T)$ be a unital association schemoid satisfying the condition $(i)$ in  
Definition \ref{defn:semi-thin}.  
If $\sigma\in{_{\alpha}S_{\beta}}$, then
$$
\sharp\left(\sigma x\right) =
\begin{cases}
1 & \text{if} \ x\in X_{\alpha}, \\
0 & \text{otherwise}. 
\end{cases}
$$
\end{lem}
\begin{proof}
By the definition of the subset ${_{\alpha}S_{\beta}}$, we have the result. 
\end{proof}

\begin{lem}\label{Lemma3}
Let $(\mathcal{C}, S, T)$ be a semi-thin association schemoid. For any $\alpha ,\ \beta ,\ \gamma\in S_0$, $\sigma\in{_{\alpha}S_{\beta}}$ and $\tau\in{_{\beta}S_{\gamma}}$, there exists a unique element $\mu=\mu(\tau, \sigma)$ in ${_{\alpha}S_{\gamma}}$ such that $p^{\mu}_{\tau\sigma}=1$. Moreover,  $p^{\mu '}_{\tau\sigma}=0$ if $\mu '\in S$ and $\mu '\not=\mu$.
\end{lem}
\begin{proof}
We show that there exists $\mu\in{_{\alpha}S_{\gamma}}$ such that $p^{\mu}_{\tau\sigma}\geq 1$. 
Let $x$ be an element in $X_{\alpha}$. In view of  Lemma \ref{Lemma2}, we see that $\sharp\left(\sigma x\right) =1$. 
Let $f$ be the unique element of $\sigma x$; that is, $\sigma x=\{ f\}$ and $t(f) \in \beta$. 
Lemma \ref{Lemma2} implies that $\tau t(f) =\{ g\}$ for some $g \in \tau$. Then there is an exactly one element 
$\mu\in{_{\alpha}S_{\gamma}}$ such that $g\circ f\in\mu$. Thus we have $p^{\mu}_{\tau\sigma}\geq 1$.

We prove that $p^{\mu}_{\tau\sigma}\leq 1$.
Let $s_1,\ s_2\in\sigma$ and $t_1,\ t_2\in\tau$ satisfying $t_1\circ s_1=t_2\circ s_2=m\in\mu$. Since $\sharp(\sigma s(m)) =1$, it follows that $s_1=s_2$. On the other hand, we see that $\{ t_1\} =\tau t(s_1) =\tau t(s_2) =\{ t_2\}$ 
since $\sharp (\tau t(s_1)) =\sharp (\tau t(s_2)) =1$. This yields that $t_1=t_2$.

We show that $\mu =\nu$ if $p^{\mu}_{\tau\sigma}=p^{\nu}_{\tau\sigma}=1$.
Let $t_1\circ s_1=m_1\in\mu$ and $t_2\circ s_2=n_2\in\nu$ where $s_1,\ s_2\in\sigma$ and $t_1,\ t_2\in\tau$. Since  $s_1=t^{-1}_1\circ m_1$, 
it follows that $p^{\sigma}_{\tau ^*\mu}\geq 1$. By the definition of the schemoid, we see that there exist $t_3\in\tau$ and $m_3\in\mu$ such that $s_2=t^{-1}_3\circ m_3$. We have $s(t_2)=t(s_2)=t(t^{-1}_3)=s(t_3)$. 
Lemma \ref{Lemma2} yields that $t_2=t_3$. This enables us to conclude that $n_2=t_2\circ s_2=t_2\circ\left( t^{-1}_3\circ m_3\right) =t_2\circ t^{-1}_2\circ m_3=m_3$. Therefore $\mu\cap\nu\not=\phi$ and hence $\mu =\nu$.
\end{proof}

Let $(\mathcal{C}, S, T)$ be a semi-thin schemoid. We define a category $\widetilde{R}(\mathcal{C}, S, T)=\mathcal{G}$ by 
 $\mathit{ob}(\mathcal{G})=S_0$ and $\mathrm{Hom}_{\mathcal{G}}(\alpha, \beta) ={_{\alpha}S_{\beta}}$, where $\alpha , \beta\in S_0$. 
For $\sigma \in \text{Hom}_\G(\alpha, \beta)$ and $\tau \in \text{Hom}_\G(\beta, \gamma)$, the composite is defined by 
$\tau\circ \sigma = \mu(\tau, \sigma)$ using the same element $\mu$ as in Lemma \ref{Lemma3}.

\begin{lem}\label{Lemma4}
Let $\sigma\in{_{\alpha}S_{\beta}}$, $\tau\in{_{\beta}S_{\gamma}}$, $f\in\sigma$ and $g\in\tau$. If $t(f)=s(g)$, then $g\circ f\in\tau\circ\sigma$.
\end{lem}
\begin{proof}
If $g\circ f\in\mu$, then $p^{\mu}_{\tau\sigma}\geq 1$. Lemma \ref{Lemma3} implies that $\mu= \tau\circ \sigma$. 
\end{proof}

\begin{prop}
$\widetilde{R}(\mathcal{C}, S, T)$ is a category.
\end{prop}

\begin{proof}
Let $x\in X_{\alpha}$, $\sigma\in{_{\alpha}S_{\beta}}$, $\tau\in{_{\beta}S_{\gamma}}$ and $\gamma\in{_{\alpha}S_{\gamma}}$. By Lemma \ref{Lemma2}, we see that $\sharp (\sigma x)=1$ and hence 
$\sigma x=\{ f\}$ with $f \in mor(\C)$. 
Moreover,  we have 
$\tau\bigl( t(f)\bigr) =\{ g\}$ with an appropriate  morphism $g$ in $\C$.  
Lemma \ref{Lemma4} implies that $h\circ (g\circ f)\in\mu\circ (\tau\circ\sigma )$ and $(h\circ g)\circ f\in (\mu\circ\tau )\circ\sigma$. Since $\bigl(\mu\circ (\tau\circ\sigma )\bigr)\cap\bigl( (\mu\circ\tau )\circ\sigma\bigr)\not=\phi$, it follows that $\mu\circ (\tau\circ\sigma )=(\mu\circ\tau )\circ\sigma$.

For $\alpha \in S_0$, we see that $\alpha \in {_\alpha S_\alpha}=\text{Hom}_{\G}(\alpha, \alpha)$. 
For $\sigma \in \text{Hom}_{\G}(\alpha, \beta)$, it follows from Lemma \ref{Lemma4} that 
$\beta \circ \sigma = \sigma = \sigma \circ \alpha$. This completes the proof.  
\end{proof}

\begin{prop}\label{TCST}
The category $\widetilde{R}(\mathcal{C}, S, T)$ is a groupoid.
\end{prop}

\begin{proof}
Suppose that $\sigma$ is in $\text{Hom}_{\G}(\alpha, \beta)$. Lemma \ref{Lemma4} yields that 
$\sigma ^*\circ\sigma =\alpha$ and $\sigma\circ\sigma ^*=\beta$. 
We have $\sigma ^{-1}=\sigma ^*$.
\end{proof}

Let $st{\bf ASmd}$ denote a full subcategory of ${\bf ASmd}$ whose objects are semi-thin association schemoids. We here construct 
a functor $\widetilde{R}( \ )$ from $st{\bf ASmd}$ to the category $\mathbf{Gpd }$ of groupoids.

Let $(\mathcal{C}, S, T)$ be a semi-thin association schemoid. It follows from Proposition \ref{TCST} that 
 $\widetilde{R}(\mathcal{C}, S, T)=\mathcal{G}$ is a groupoid. Let $F$ be a morphism between semi-thin association schemoids $(\mathcal{C}, S, T)$ 
 and $(\mathcal{C}', S', T')$. By definition, for any $\sigma\in S=\mathit{mor}(\mathcal{G})$, there exists a unique element 
 $\tau\in S'=\mathit{mor}
 (\mathcal{G'})$ such that $F(\sigma )\subset\tau$. Since $\alpha\in\coprod_{x\in\mathit{ob}(\mathcal{C})}\{ 1_x\}$ 
 for any $\alpha\in S_0=\mathit{ob}
 (\mathcal{G})$, there exists a unique element $\beta\in S'_0=\mathit{ob}(\mathcal{G}')$ 
 such that $F(\alpha )\subset\beta$. 
We then define a functor  $\widetilde{R} : st{\bf ASmd} \to \mathbf{Gpd}$ by $\widetilde{R}(F)(\alpha )=\beta$ and $\widetilde{R}(F)(\sigma )=\tau$.

\begin{defn}\label{defn:thin_schemoid}
A semi-thin association schemoid $(\mathcal{C}, S, T)$ is a {\it thin} association schemoid with a subset $V$ of base points of 
$ob (\C)$ if 

(iii) $\sharp\mathrm{Hom}_{\mathcal{C}}\left( x,\ y\right)\leq 1$ for $x,\ y\in ob(\C)$ and 

(iv) the subset $V \subset ob(\C)$ satisfies the condition that 
for any connected component $C$ of $ob (\C)$, $\sharp (C\cap V) =1$ and 
 the map $\varphi :V\rightarrow S_0$ defined by $\varphi (v)\ni 1_v$ is bijective.
\end{defn}

Let  $t{\bf ASmd}$ be the full subcategory  
of $ st{\bf ASmd}$ whose objects are thin association schemoids. 
We have a commutative diagram of categories and functors 
$$
\xymatrix@C35pt@R15pt{
&& {\bf ASmd}\\
&& st{\bf ASmd} \ar@<2pt>[dll]^(0.3){\widetilde{R}(\ )} \ar@{^{(}->}[u]\\
\mathbf{Gpd} \ar@<2pt>[rr]^(0.6){\widetilde{S}(\ )}  \ar@<2pt>[urr] \ar[uurr]^{\widetilde{S}(\ )} && t{\bf ASmd}. \ar@<2pt>[ll]^{\widetilde{R}(\ )} \ar@{^{(}->}[u]
}
$$

\begin{rem}
In \cite{H-Y}, Hanaki and Yoshikawa give a procedure to make a groupoid with a thin coherent configuration as an ingredient.   
The construction factors through $t{\bf ASmd}$ the category of thin association schemoids; 
see Remark \ref{rem:categories} (i).  
\end{rem}


Let $(\C, S, T)$ be a thin association schemoid with base points. We here define 
functors $\Phi :(\C, S, T)\rightarrow \widetilde{S}\widetilde{R}(\C, S, T)$ and  
$\Psi : \widetilde{S}\widetilde{R}(\C, S, T)\rightarrow(\C, S, T)$. Moreover we shall prove 

\begin{prop}\label{prop:isomorphism} Let $(\C, S, T)$ be a thin association schemoid with a set $V$ of base points.
Then the functor $\Phi :(\C, S, T)\rightarrow \widetilde{S}\widetilde{R}(\C, S, T)$ is an isomorphism with 
the inverse $\Psi$. Moreover,  $\Phi$ preserves the sets of base points. 
\end{prop}

Thus we have the main result in this section. 

\begin{thm}\label{thm:eq_categories} {\em (}cf. \cite[Proposition 5.2]{H}{\em )} The functor $\widetilde{S}( \ )$ gives rise to an equivalence between 
the category $\mathbf{Gpd}$ of groupoids and 
the category $(t{\bf ASmd})_0$ of based thin association schemoids.  
Moreover, the functor $\widetilde{R}( \ ) : (t{\bf ASmd})_0 \to \mathbf{Gpd}$ is the right adjoint for  
$ \widetilde{S}( \ ) : \mathbf{Gpd} \to (t{\bf ASmd})_0$. 
\end{thm}

\begin{proof}
The results follow from Corollary \ref{cor:fully_faithful} and Proposition \ref{prop:isomorphism}.
\end{proof}

In order to define the functor $\Phi$ mentioned above, we recall the condition (iii) in Definition \ref{defn:thin_schemoid}. Then for any 
object $x \in ob(\C)$, we see that there are an exactly one element $v \in V$ and  
a unique morphism $\rho_x$ in $\C$ such that 
$\text{Hom}_\C(x, v) = \{\rho_x\}$. 
Moreover, we choose the partition $\sigma_x \in S$ so that $\rho_x$ is in $\sigma$. Then define a functor 
$\Phi : (\C, S, T)\rightarrow \widetilde{S}\widetilde{R}(\C, S, T)=(S, \{S_g\}_{g \in S}, T')$ 
by 
$\Psi(x) = \sigma_x$ for $x \in ob (\C)$ and 
$$
\Phi\begin{pmatrix}\xymatrix@C20pt@R15pt{
x\ar[rr]^f\ar[rd]_{\rho _x}&&y\ar[ld]^{\rho _y}\\&v&}\\ 
\end{pmatrix} =\left(\sigma _y, \sigma _x\right)
$$
for $f \in mor(\C)$. In order to define a functor from $\widetilde{S}\widetilde{R}(\C, S, T)$ to $(\C, S, T)$, we need the following fact. 

\begin{claim}\label{claim:1}
$\sharp \varphi^{-1}(\beta) \sigma = 1$. 
\end{claim}

\begin{proof}
Suppose that $f_\sigma$ and $g_\sigma$ are in $\varphi^{-1}(\beta)\sigma$. There exists a unique partition $\tau \in S$ such that 
$T(\sigma) \subset \tau$. Then $f_\sigma^{-1}$ and $g_\sigma^{-1}$ are in $\tau$. It follows that 
$s(f_\sigma^{-1}) = t(f_\sigma)=\varphi^{-1}(\beta) = t(g_\sigma)=s(g_\sigma^{-1})$. The condition (i) 
in Definition \ref{defn:semi-thin} implies that  $f_\sigma^{-1} = g_\sigma^{-1}$.  
\end{proof}

We define a functor $\Psi : \widetilde{S}\widetilde{R}(\C, S, T)\rightarrow(\C, S, T)$ by 
$\Psi(\sigma) = s(f_\sigma)$ and 
$$\Psi (
\xymatrix@C15pt@R10pt{\sigma\ar[rr]^{\left(\tau, \sigma\right)} 
&&\tau} 
)
=\begin{matrix}\xymatrix@C10pt@R15pt{s(f_{\sigma})\ar[rr]^{f^{-1}_{\tau}f_{\sigma}}\ar[rd]_{f_{\sigma}}&&s(f_{\tau})\ar[ld]^{f_{\tau}}\\&\varphi ^{-1}
(\beta ),&}\\ \end{matrix}
$$ 
where $t(\sigma) = \beta$ and $\varphi^{-1}(\beta)= \sigma$.

\begin{proof}[Proof of Proposition \ref{prop:isomorphism}] By definition, it is readily seen that $\Phi$ and $\Psi$ are functors. 
We prove that $\Psi$ is an isomorphism of schemoids preserving the set of base points. 

For any object $x$ in $\C$, we see that $\Psi\Phi(x) = \Psi(\sigma) = s(f_\sigma)$, where $\text{Hom}_\C(x, v) = \{\rho \}$, $\rho \in \sigma$ for some $v \in 
V$, $\sigma \in {_\alpha S_\beta}$ and $\varphi^{-1}\sigma = \{f_\sigma \}$. Since $1_v \in \beta$, it follows that $\varphi^{-1}(\beta) = v$. 
Claim \ref{claim:1} yields that $\rho = f_\sigma$ and hence $s(f_\sigma) = x$. 

Let $\sigma$ be an object in $\widetilde{S}\widetilde{R}(\C, S, T)$; that is, $\sigma \in S$. Then 
$\Phi\Psi(\sigma) =  \Phi(s(f_\sigma))=\sigma$ because $f_\sigma \in \sigma$. 

Observe that $(\C, S, T)$ and $\widetilde{S}\widetilde{R}(\C, S, T)$ are thin. The condition (iii) in Definition \ref{defn:thin_schemoid} enables us to 
conclude that the functors $\Psi\Phi$ and $\Phi\Psi$ are identity on the set of morphisms since so are on objects. 

We prove that $\Phi$ and $\Psi$ preserve partitions.  For any $\sigma \in S$, let $f : x \to y$ be a morphism in $\sigma$. 
Suppose that $\Phi(f)= (\sigma_y, \sigma_x)$ and $\sigma_y^*\sigma_x = \tau$. 
It follows from Lemma \ref{Lemma4} that $f = \rho_y^{-1}\rho_x \in \tau$. Thus we see that $\tau = \sigma$ and hence 
$\Psi(\sigma) \subset S_\sigma$. 
By definition, we see that $\xymatrix@C18pt@R15pt{\Psi(\sigma\ar[r]^{\left(\tau, \sigma\right)} 
&\tau)}  =  \xymatrix@C18pt@R15pt{s(f_{\sigma})\ar[r]^{f^{-1}_{\tau}f_{\sigma}}&s(f_{\tau})}.
$
Suppose that $\tau^*\circ \sigma = \mu$. Then $f^{-1}_{\tau}f_{\sigma}\in \tau^*\circ \sigma = \mu$. Thus we have 
$\Psi(S_\mu) \subset \mu$. 

In order to prove that $\Phi$ preserves the set of base points, we take an element $v$ in $V$. Then it follows that 
$\text{Hom}_\C(v, v) = \{ 1_v\}$, $1_v \in \varphi(v)$ and hence $\Phi(V) \subset S_0$; see Definition \ref{defn:thin_schemoid}. The map 
$\varphi : V \to S_0$ is a bijection by definition. We have $\Phi(V) = S_0$. This completes the proof. 
\end{proof}

We conclude this section with an example of a semi-thin association shcemoid $(\mathcal{C}, S, T)$ which is not 
isomorphic to $\widetilde{S}\widetilde{R}(C, S, T)$. 

\begin{ex} Let $I$ be a set with an element $1$. 
For any $i \in I$, let $\C_i$ be a groupoid of the form 
$$
\xymatrix{
x_i \ar@(lu,ld)_{1_{x_i}} \ar@<+2pt>[r]^{f_i} & y_i, \ar@(ru,rd)^{1_{y_i}} \ar@<+2pt>[l]^{g_i}
}
$$
and $\C_I$ the disjoint union of the categories $\C_i$ over $I$.  
Define a partition $S$ of $mor (\C_I)$ by $S=\{\sigma_0^1, \sigma_0^2, \tau_1, \tau_2\}$, where 
$\sigma_0^1=\{1_{x_i}\}_{i\in I}$,  $\sigma_0^2=\{1_{y_i}\}_{i\in I}$, $\tau_1=\{f_i\}_{i\in I}$ and $\tau_2=\{g_i\}_{i\in I}$. 
Moreover, we define a contravariant functor $T$ with $T^2 = 1_{\C_I}$ on $\C_I$ by 
$T(x_i) = y_i$ and $T(g_i) = f_i$. Then $(\C_I, S, T)$ is a schemoid. A direct computation shows that 
$\widetilde{R}(C_I, S, T)\cong \C_1$ for any $I$ but $\widetilde{S}\C_1\cong (C_I, S, T)$ if and only if 
$\sharp I =2$. In fact, $(C_I, S, T)$ is not thin if $\sharp I >2$; see (iv) in Definition \ref{defn:thin_schemoid}.
\end{ex}

\section{Extensions of schemoids}
In order to assert that the categories $q{\bf ASmd}$ and ${\bf ASmd}$ are more fruitful, it is important to construct (quasi-)schemoids systematically. 
This section contributes to it.       
We begin with the definition of a linear extension of a small category in the sense of Baues and Wirsching \cite{B-W}. 

Let $F(\C)$ be the category of factorizations in $\C$; that is, the objects are the morphisms in $\C$ and morphisms $f \to g$ are the pairs 
$(\alpha, \beta)$ for which  
diagram 
$$
\xymatrix@C35pt@R15pt{
t(f) \ar[r]^{\alpha} & t(g) \\
s(f) \ar[u]^f& s(g) \ar[u]_g \ar[l]^{\beta}
}
$$
commutes. 
The composition is defined by $(\alpha', \beta')\circ (\alpha, \beta)=(\alpha'\alpha, \beta\beta')$. 

\begin{defn}\label{defn:extensions}(\cite[(2.2) Definition]{B-W})
Let $\C$ and $\E$ be small categories.  
Let $D : F(\C) \to \K\text{-Mod}$ denote a natural system, namely a functor from $F(\C)$ to the category of $\K$-modules.  We say that 
$$
D_+ \to {\mathcal E} \stackrel{q}{\to} \C
$$
is a linear extension if (a), (b) and (c) hold: 

(a) $\E$ and $\C$ have the same objects and $q$ is a full functor which is the identity on objects. 

(b) For each morphism $f : A \to B$ in $\C$,  the abelian group $D_f$ acts transitively and effectively on the subset $q^{-1}(f)$ 
of morphisms in $\E$. We write $f_0+\alpha$ for the action of $\alpha \in D_f$ on $f_0 \in q^{-1}(f)$. 

(c) The action satisfies the {\it linear distributivity law}:
$$
(f_0+\alpha)(g_0+\beta) = f_0g_0+f_*\beta + g^*\alpha, 
$$ 
where $f_*=D(f, 1)$ and $g^*=D(1, g)$.
\end{defn}

We give a linear extension a schemoid structure under appropriate assumptions. 

\begin{prop}\label{prop:extension}
Let 
$
D_+ \to {\mathcal E} \stackrel{q}{\to} \C
$
be a linear extension of $\C$.  
Let $(\C, S)$ be a quasi-schemoid. 
Assume further that  for any morphism $f$ in $\C$, the homomorphism $f^*$ and $f_*$ are invertible and that 
$D_{1_{s(f)}}\cong D_{1_{s(g)}}$ for any $\sigma \in S$ and $f, g \in \sigma$. 
Then ${\mathcal E}$ admits a unique quasi-schemoid structure for which $q$ is a morphism of 
quasi-schemoids and injective on the partition of $mor ({\mathcal E})$.
\end{prop}

We call such a morphism $q$ in Proposition \ref{prop:extension} a {\it proper} morphism. 

\begin{rem}Let $\C$ be a quasi-schemoid. 
If $\C$ has a connected groupoid structure, for example objects in $j({\bf AS})$ and connected objects in 
$\widetilde{S}({\bf Gpd})$, then all the assumptions in Theorem \ref{prop:extension} and 
in Theorem \ref{thm:extension_schemoid} below are satisfied. 
\end{rem}

\begin{lem}\label{lem:proper_map}
Let $q : ({\mathcal E}, \widetilde{S})  \to (\C, S)$ be a proper morphism of quasi-schemoids if and only if  
$\widetilde{S}=\{ q^{-1}(\sigma)\}_{\sigma \in S}$. 
\end{lem}
\begin{proof}
It is immediate. 
\end{proof}

\begin{proof}[Proof of  Proposition \ref{prop:extension}]
Let $H$ be a partition of $mor({\mathcal E})$ defined by the sets $q^{-1}(\sigma)$ for $\sigma \in S$.  
We prove that $H$ satisfies the concatenation axiom in Definition  \ref{defn:schemoid}. 
Consider a commutative diagram 
$$
\xymatrix@C25pt@R15pt{
q^{-1}(\tau)\times_{\mathcal E}q^{-1}(\mu) & \widetilde{\pi_{\tau\mu}}^{-1}(q^{-1}(\sigma)) \ar[r]^-q \ar[l]_-{\supset} 
\ar[d]_{\widetilde{\pi_{\tau\mu}^\sigma}} & 
{\pi_{\tau\mu}}^{-1}(\sigma) \ar[r]^-{\subset} \ar[d]^{{\pi_{\tau\mu}^\sigma}} & \tau \times_{\mathcal C} \mu \\
& q^{-1}(\sigma) \ar[r]_{q} & \sigma, & 
}
$$
where $\pi_{\tau\mu}$ and $\widetilde{\pi_{\tau\mu}}$ are the maps defined by the concatenation of morphisms and 
$\pi_{\tau\mu}^\sigma$ and $\widetilde{\pi_{\tau\mu}^\sigma}$ denote the restrictions of $\pi_{\tau\mu}$ and $\widetilde{\pi_{\tau\mu}}$, respectively. 
We take a morphism $f^0 \in q^{-1}(f)$ for each $f$. Let $+ : D_f\times D_f \to D_f$ denote the sum on $D_f$. 
For any $f^0+\gamma$ in $q^{-1}(\sigma)$, we define  maps between sets  
$$
\xymatrix@C40pt@R15pt{
(\pi_{\tau\mu}^\sigma)^{-1}(f)\times (+)^{-1}\{\gamma\} \ar@<0.5ex>[r]^-{\theta}&  
\ar@<0.5ex>[l]^-{\xi}(\widetilde{\pi_{\tau\mu}^\sigma})^{-1}(f^0+\gamma) 
}
$$by 
\begin{eqnarray*}
\theta(f_i, f_j, \alpha, \beta) &=& (f_i^0-(f_j^*)^{-1}u^0 + (f_j^*)^{-1}\alpha, f_j^0+((f_i)_*)^{-1}\beta)  \ \ \text{and}  
\\
\xi(f_i^0+\alpha', f_j^0+\beta') &=& (q(f_i^0), q(f_j^0), u^0+(f_j)^*(\alpha'), ((f_i)_*)\beta'), 
\end{eqnarray*}
where $u^0$ is an element in $D_f$ which is uniquely determined by the equality 
$f_i^0\circ f_j^0 = f^0 + u^0$. We see that $\theta$ is a bijection with inverse $\xi$.  Moreover, we have 
$(+)^{-1}\{\gamma\} \cong D_f \cong D_{1_{s(f)}}$ for any $\gamma$ since, by assumption, 
$f_* : D_{1_{s(f)}} \to D_f$ is invertible for $f \in \sigma$. 
The assumption that $D_{1_{s(f)}}\cong D_{1_{s(g)}}$ for any $f, g \in \sigma$ implies that the cardinal number of 
$(+)^{-1}\{\gamma\}$ does depend on only the choice of $\sigma \in S$. 
This enables us to conclude that $(\mathcal E, H)$ is a quasi-schemoid. It is immediate that $q$ is a proper morphism.  

The uniqueness of the schemoid structure on ${\mathcal E}$ follows from Lemma \ref{lem:proper_map}. 
This completes the proof. 
\end{proof}

Let $\C$ be a small category. Let $\pi : F(\C) \to \C^{op}\times \C$ denote the natural functors \cite[(1.16)]{B-W} defined by 
$\pi(f) = (s(f), t(f))$ for $f \in ob(F(\C))$ and $\pi(\alpha, \beta) =(\beta, \alpha)$ for $(\alpha, \beta) \in mor (F(\C))$. 
Let $p :  \C^{op}\times \C \to \C$  be the obvious forgetful functor. 
 Proposition \ref{prop:extension} and its proof deduce the following result.   

\begin{thm} \label{thm:extension_schemoid}
Let $H: \C \to \K\text{\em -Mod}$ be a functor and $D$ the natural system induced by $H$, namely 
$D=\pi^*p^*H$. Let  $D_+ \to \E \stackrel{q}{\to} \C$ be a linear extension. Suppose that 
$(\C, S, T)$ is an association schemoid whose underlying category  $\C$ is a connected groupoid 
with $T(f) = f^{-1}$ for  $f\in mor(\C)$. Then 
the quasi-schemoid $(\E, \widetilde{S})$ defined in Proposition \ref{prop:extension} admits a schemoid structure for which $\E$ is a groupoid with 
$\widetilde{T}(\widetilde{f})=\widetilde{f}^{-1}$, and $q$ is a morphism of association schemoids. 
Moreover for $\sigma, \tau$ and $\mu$ in $\widetilde{S}$, one has $p_{\tau\mu}^\sigma = \sharp D_g p_{q(\tau)q(\mu)}^{q(\sigma)}$ 
for some $g$ and hence any $g$. 
\end{thm}

A linear extension $D_+ \to \E \stackrel{q}{\to} \C$ in Proposition \ref{prop:extension} 
or Theorem \ref{thm:extension_schemoid} is called a {\it schemoid extension} of $(\C, S)$. 
 
We prove Theorem \ref{thm:extension_schemoid} by using the following lemma. 

\begin{lem}\label{lem:lem1} Let $\Delta$ be a normalized cocycle in $F^2(\C; D)$; see \cite[(1.9) Theorem]{B-D}.   
With the same assumption as in Theorem \ref{thm:extension_schemoid}, one has
$
f_*\Delta(f^{-1}, f) = f^*\Delta(f, f^{-1}). 
$
\end{lem}

\begin{proof}
We see that $$0 = (\delta \Delta)(f, f^{-1}, f) = f_*\Delta(f^{-1}, f) -\Delta(1, f) + \Delta(f, 1) - f^*\Delta(f, f^{-1}). $$ 
Observe that $\Delta(1, g) = 0 = \Delta(g, 1)$ because $\Delta$ is normalized cocycle. 
This completes the proof. 
\end{proof}

\begin{proof}[Proof of Theorem \ref{thm:extension_schemoid}]
We verify that for the partition $H$ 
the condition (i) in Definition  \ref{defn:association_schemoids} holds. Since $q$ is the identity map on objects, 
it follows that $$q^{-1}(\amalg_{y\in ob\C}\text{Hom}_\C(y, y))=\amalg_{x\in ob\E}\text{Hom}_\E(x, x)=:J.$$
We see that for $q^{-1}(\sigma) \in H$, if $q^{-1}(\sigma) \cap J \neq \phi$, then 
$\sigma \cap  \amalg_{y\in ob\C}\text{Hom}_\C(y, y) \neq \phi$. Thus $\sigma \subset \amalg_{y\in ob\C}\text{Hom}_\C(y, y)$ 
as $\C$ is a quasi-schemoid. This yields that $q^{-1}(\sigma)\subset J$. 

Without loss of generality, we assume that 
$(g, \beta)\circ (f, \alpha)= (gf, -\Delta(g,f) + g_*\alpha + f^*\beta)$
for morphisms $(f, \alpha)$ and $(g, \beta)$ in ${\mathcal E}$, 
where $\Delta$ is a normalized cocycle; 
see \cite[(1.9) Theorem]{B-D} and the proof of \cite[(2.3) Theorem]{B-W}. 
Observe that $D=\pi^*p^*D_M$ and hence  $f^*= id$.  We then have $(f, \alpha)\circ (1, 0) = (f, \alpha) = 
(1, 0) \circ (f, \alpha)$. For an element $(f, \alpha)$, we define 
$$
(f, \alpha)^{-1}= (f^{-1}, \beta) = (f^{-1}, (f_*)^{-1}(-\alpha + \Delta(f, f^{-1})). 
$$
Then we see that 
\begin{eqnarray*}
(f, \alpha)\circ (f^{-1}, \beta)&=& (ff^{-1}, -\Delta(f, f^{-1}) + f_*(f_*^{-1}(-\alpha + \Delta(f,f^{-1}) + (f^{-1})^*\alpha) \\
&=&(1, 0)  \ \ \ \text{and that} \\
(f^{-1}, \beta)\circ (f, \alpha) &=& (f^{-1}f, -\Delta(f^{-1},  f) + (f^{-1})_*\alpha + 
f^*((f^{-1})_*(-\alpha + \Delta(f, f^{-1}))) \\
&=& (1, -\Delta(f^{-1}, f) + (f^{-1})_*\Delta(f, f^{-1})) =(1, 0).
\end{eqnarray*}
The last equality follows from Lemma \ref{lem:lem1}. 

We define a natural transformation $T : \E \to \E$ by $\widetilde{T}(\widetilde{f})=\widetilde{f}^{-1}$. 
Then it follows that $\widetilde{T}$ is a contravariant functor 
and $\widetilde{T}^2= id$. 
It is readily seen that $Tq=q\widetilde{T}$ and hence    
$q^{-1}(\sigma)^* :=\{\widetilde{T}(\widetilde{f}) \mid \widetilde{f} \in q^{-1}(\sigma) \}
= q^{-1}(\sigma^*).
$
We have the result. 
\end{proof}


We recall the definition of an equivalence between linear extensions. Two linear extensions  
$D_+ \to {\mathcal E} \stackrel{q}{\to} \C$ and $D_+ \to {\mathcal E} \stackrel{q'}{\to} \C$ are {\it equivalent} if there exists an isomorphism 
$\varepsilon : {\mathcal E} \stackrel{\cong}{\to} {\mathcal E}'$ of categories with $q'\varepsilon = q$ and with 
$\varepsilon(f_0+\alpha) = \varepsilon(f_0) + \alpha$ for $f_0 \in mor({\mathcal E})$ and $\alpha \in D_{qf_0}$; see \cite[page 193]{B-W}. 
A linear extension ${\mathcal E} \stackrel{q}{\to} \C$ is {\it split} 
if there exists a functor $s : \C  \to {\mathcal E}$ such that $qs = 1$.

For an $\mathbb{Z}$-module $M$, we define a natural system $\underline{M} : F(\C)  \to \mathbb{Z}\text{-Mod}$, 
which is so-called the trivial representation, by   
$\underline{M}(x)= M$ and $\underline{M}(f) = id_{M}$ for $x \in ob(\C)$ and $f \in mor(\C)$.  
Then the Baues-Wirsching cohomology 
$H_{BW}^*(\C, D)$ is defined by 
$$
H_{BW}^*(\C, D) = {\rm Ext}_{{\bf Func}(F(\C),  \mathbb{Z}\text{-Mod})}^{\ *}(\underline{\mathbb{Z}}, D). 
$$
Originally, the Baues-Wirsching cohomology is defined by using a cochain complex; see \cite[(1.4) Definition]{B-W}. 
The result \cite[(4.4) Theorem]{B-W} allows us to obtain the extension functor description mentioned above.  

A functor $F: \C \to \mathbb{Z}\text{-Mod}$ induces a natural system 
$\pi^*p^*F : F(\C) \to  \mathbb{Z}\text{-Mod}$. 
Then we see that  the cohomology 
$H^*(\C, F):={\rm Ext}_{{\bf Func}(\C, \K\text{-Mod})}^*(\underline{\mathbb{Z}}, F)$ is isomorphic to 
the Baues-Wirsching cohomology 
$H_{BW}^*(\C, \pi^*p^*F)$; see \cite[(8.5) Theorem]{B-W}.  In particular, for the trivial representation $\underline{\mathbb{Z}}$, we have 
$H_{BW}^*(\C, \underline{\mathbb{Z}})\cong H^*(\C, \underline{\mathbb{Z}})\cong H^*(B\C; \mathbb{Z}),$ 
where the last one is the singular cohomology of the classifying space of $\C$.

The result \cite[(2.3) Theorem]{B-W} implies that the second Baues-Wirsching cohomology classifies 
linear extensions over a small category.  
Thus Proposition \ref{prop:extension} enables us to deduce the following result. 

\begin{thm}\label{classification_ex}
Let $(\C, S)$ be a quasi-schemoid and $D : F(\C) \to \mathbb{Z}\text{-Mod}$ a natural system for which $f_*$ and $f^*$ are invertible 
for any $f \in mor(\C)$ and $D_{1_{s(f)}}\cong D_{1_{s(g)}}$ for any $\sigma \in S$ and $f, g \in \sigma$. 
Then the Baues-Wirsching cohomology $H^2_{BW}(\C; D)$ classifies schemoid extensions of the form 
$D_+ \to \E \stackrel{q}{\to} \C$ with $q$ proper.  
\end{thm}


Let $(X, S)$ be an association scheme. 
Since the underlying category of $j(X, S)$ is a directed complete graph, we see that the category is equivalent to the trivial category $\bullet$, namely 
the category consisting of one object and one morphism. 
It is immediate that $H_{BW}^*(\bullet; D)=0$ for $* >0$ and any natural system $D$ on $\bullet$.
Then the result \cite[(2.3) Theorem]{B-W} deduces the following corollary. 

\begin{cor}\label{cor:trivial_extensions} Every schemoid extension of $j(X, S)$ 
is split.  
\end{cor}

\begin{ex}\label{ex:contractible_spaces}
Theorem \ref{classification_ex} asserts that the classification of schemoid extensions is reduced in turn 
to that of extensions of the underlying category. We here comment on such linear extensions of categories. 
   
Let $M$ be an abelian group and $G$ a finite group. 
Let $\underline{M} : \C \times G \to \mathbb{Z}\text{-Mod}$ denote the trivial representation of $\C\times G$. 
We also write $\underline{M}$ for the induced representation 
$\pi^*p^*\underline{M} : F(\C\times G) \to  \mathbb{Z}\text{-Mod}$. 
Suppose that the classifying space $B\C$ is contractible and that $\C$ is finite. 
For example the classifying space of the underlying category 
${\mathcal D}$ of $\C_{[k]}$ in Example \ref{ex4} is contractible because ${\mathcal D}$ has the initial object. 
Then every linear extension $\underline{M}_+ \to \E \stackrel{q}{\to} \C\times G$ is isomorphic to 
a linear extension of the form $\underline{M}_+ \to \C \times K \stackrel{1\times q'}{\to} \C\times G$ which  is induced by some extension 
${M} \to K \stackrel{q'}{\to} G$ of the group $G$. In fact, since $\C$ and $G$ are finite, 
it follows that $B(\C\times G) \cong B\C\times BG$. Then we see that 
$H^2_{BW}(\C \times G; \underline{M}) \cong H^2(B\C\times BG; M)\cong H^2(BG; M)$. More precisely, linear extensions of $\C\times G$ associated with the natural system $\underline{M}$ is classified by the group homology $H^2(G; M)$. 
\end{ex}

\begin{ex}\label{ex:blow-ups}
Let $(X, S)$ be an association scheme. The same argument as in Example \ref{ex:contractible_spaces} 
enables us to deduce that there exists a non-split schemoid extension of  
$j(X, S)\times (\mathbb{Z} \times \mathbb{Z})^\bullet$ while $j(X, S)$ and the first extension  $j(X, S)\times (\mathbb{Z})^\bullet$ 
have no non-split schemoid extension; see Example \ref{ex5}.  
On the other hand, the first extension  $j(X, S)\times (\mathbb{Z}/2)^\bullet$ has only one non-trivial extension 
since $H^2(B(\mathbb{Z}/2); \mathbb{Z}/2) 
\cong H^2({\mathbb R}P^\infty ; \mathbb{Z}/2)\cong {\mathbb Z}/2$. 
\end{ex}

\section{Basic schemoids and admissible morphisms}\label{section:6}

We introduce a subcategory ${\mathcal B}$ of the category of quasi-schemoids for which 
the correspondence sending each object in ${\mathcal B}$ to the schemoid algebra 
is a well-defined functor to the category of (possibly nonunital) algebras.   

\begin{defn} (cf. Definition \ref{defn:semi-thin}) \label{defn:basic} 
A unital quasi-schemoid $(\C, S)$ is {\it basic} if 
the underlying category $\mathcal{C}$ is a groupoid.  
\end{defn}

Every coherent configuration is considered as a complete graph and hence it is a basic association schemoid.  
In general, a morphism of schemes does {\it not} induce a homomorphism between their Bose-Mesner algebras. 
To overcome this inconvenience, French \cite{F} has introduced the notion of  admissible morphisms 
in the category ${\bf AS}$ of association schemes. Generalizing the notion, we introduce admissible maps 
in the category $q{\bf ASmd}$.

\begin{defn}\label{defn:admissible_maps} A morphism $\phi : (\C, S) \to ({\mathcal D}, T)$ of quasi-schemoids is 
{\it admissible} if 
for any $x \in ob(\C)$, $\sigma \in S$ and $g \in \phi(\sigma)$ with $t(g) = \phi(x)$, there exists a morphism 
$f \in \sigma$ such that $t(f) = x$ and $\phi(f) = g$.  
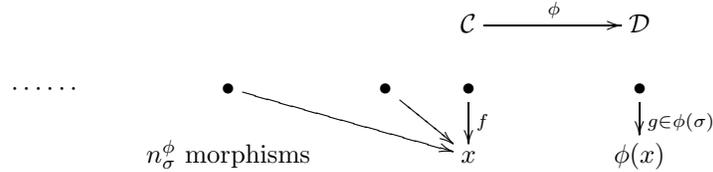
$$
\xymatrix@C20pt@R12pt{& & &  \C \ar[rr]^{\phi} & & {\mathcal D} \\
\cdots \cdots& \bullet  \ar[drr]& \bullet \ar[dr]& \bullet  \ar[d]^f&  & \bullet \ar[d]^{g \in \phi(\sigma)}\\
 &  n_\sigma^\phi \ \text{morphisms}  &     &  x    & & \phi(x) 
}
$$
\end{defn}

\begin{lem}\label{lem:composite}
The composite of any two admissible morphisms is admissible.  
\end{lem}

For elements $\sigma$ and $\tau$ of a quasi-schemoid $S$ over an category $\C$, we define 
a subset $\sigma\tau$ of the power set $2^{mor(\C)}$ by  
$
\sigma\tau = \{ \mu \in S \mid p_{\sigma\tau}^\mu \geq 1 \}. 
$
The notion of the {\it closed subset} of an association scheme is generalized to that of schemoids by the natural way. 
However, as mentioned in the Introduction, we do not deal with such generalized one in this paper. 
 
\begin{lem}\label{lem:closed_set}
If $\tau \in \phi(\rho)\phi(\pi)$ for some $\pi$ and $\rho$ in $S$, then $\tau \in \phi(S)$. 
\end{lem}

The proofs of Lemmas \ref{lem:composite} and  \ref{lem:closed_set} proceed verbatim as in those of 
\cite[Lemmas 3.3, 4.1]{F}. Two lemmas below are rewritten versions of \cite[Lemmas 3.11 and 6.3]{F}.

\begin{lem}\label{lem:admossible_key1}
Let $\phi : (\C, S) \to ({\mathcal D}, T)$ be an admissible morphism from a finite quasi-schemoid, whose 
underlying category is a groupoid, to a basic schemoid.  
Then for any $\sigma \in S$, there exists a positive integer $n_\sigma^\phi$ such that for any 
$x \in ob(\C)$ and $g \in \phi(\sigma)$ with $t(g) = \phi(x)$, 
$$
\sharp (\phi^{-1}(g)\cap x\sigma) = n_\sigma^\phi .
$$ 
\end{lem}

\begin{proof}
We define a subset $\text{Ker} \phi$ of $2^{mor(\C)}$ by 
$$
\text{Ker}  \phi = \{\kappa \in S \mid \phi(\kappa) \subset \alpha \subset J_0 \ \text{for some} \ \alpha \in T \}. 
$$
Since the morphism $\phi$ is admissible, by definition, there exists a morphism $f$ in $\C$ such that 
$t(f) = x$ and $\phi(f) = g$. We define a subset $U_\kappa$ of $mor(\C)$ by
$$ U_\kappa = \{ f' \mid (u, f') \in \pi_{\kappa\sigma}^{-1}(f) \  \text{for some} \ u \in \kappa \}. $$
Observe that for any element $f'$ in $U_\kappa$, an element $u$ above is determined uniquely 
by $f$ and $f'$ if any because $\C$ is a groupoid. 
The usual argument deduces that 
$$
\phi^{-1}(g)\cap x\sigma = \bigcup_{\kappa \in \text{Ker}  \phi} U_\kappa. 
$$
Moreover, it follows that $U_\kappa \cap U_{\kappa'} = \phi$ if $\kappa \neq \kappa'$. 
Thus we have $\sharp ( \phi^{-1}(g)\cap x\sigma)  = \sum_{\kappa \in  \text{Ker} \ \phi} \sharp U_\kappa$.  
It is immediate that the cardinal of the set $U_\kappa$ is not depend on the choice of $f$ and is only depend on 
$\kappa \in \text{Ker} \ \phi$. The sum in the right-hand side is nothing but the integer $n_\sigma^\phi$ we require. 
This completes the proof.  
\end{proof}

\begin{lem}\label{lem:admissible_key2}{\em (}cf. \cite[Lemma 6.3]{F}{\em)} 
Let  $\phi : (\C, S) \to ({\mathcal D}, T)$ be the same  admissible morphism as in Lemma \ref{lem:admossible_key1}. 
Then for any $\pi, \rho \in S$ and $\tau\in T$, one has 
$$
\sum_{\sigma : \phi(\sigma)=\tau}p_{\pi\rho}^\sigma n_\sigma^\phi = p_{\phi(\pi)\phi(\rho)}^\tau n_\pi^\phi n_\rho^\phi. 
$$
\end{lem}

\begin{proof}
The proof proceeds along the same line as that of the proof of \cite[Corollary 6.3]{F}. 
In order to obtain the result, it suffices to show that 
the equality holds when $\tau = \phi(\nu)$ for some $\nu \in S$.  
This follows from Lemma \ref{lem:closed_set}. We choose $f \in \nu$ and write $g = \phi(f)$. 
We define a subset $\Omega_0$ of $\rho$ by 
$$
\Omega_0 = \{\widetilde{f} \in \rho \mid s(\phi(\widetilde{f})) = s(\phi(f)), t(f') = t(f)  \ \text{for some} \ f' \in \pi \ 
\text{with} \ (*) \}, 
$$
where $(*)$ denotes the condition that $s(f') = t(\widetilde{f})$ and $\phi(\widetilde{f}f') = g$. 
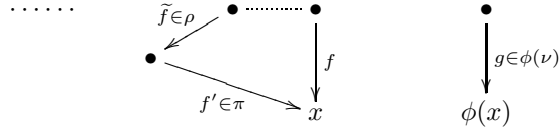
$$
\xymatrix@C20pt@R8pt{
\cdots \cdots&  & \bullet \ar@{..}[r] \ar[dl]_-{\widetilde{f} \in \rho}& \bullet  \ar[dd]^-f &  & \bullet \ar[dd]^-{g \in \phi(\nu)}\\
&  \bullet   \ar[drr]_-{f' \in \pi}& &                   & &  \\
 &    &     &  x   & & \phi(x) 
}
$$
For any $\widetilde{f}$ in $\Omega_0$, there exists $\sigma \in S$ such that $f'\widetilde{f}$ is in $\sigma$. 
We see that $\phi(f'\widetilde{f}) \in \phi(\nu)$ by definition and hence $\phi(\sigma) = \phi(\nu)$. 
On the other hand, Lemma \ref{lem:admossible_key1} enables us to deduce that 
for any $\sigma \in S$ with $\phi(\sigma) = \phi(\nu)$, there exist $n_\sigma^\phi$ elements 
$\alpha$ in $\phi(g) \cap x\sigma$, where $x = s(f)$. For each of these choices of $\alpha$, 
there exist $n_{\pi\rho}^{\sigma}$ elements $\widetilde{f} \in \rho$ such that $\widetilde{f}$ is in $\Omega_0$. 
We then conclude that 
$$
\sharp \Omega_0 = \sum_{\sigma \in S: \phi(\sigma) = \phi(\nu)}n_{\pi\rho}^{\sigma} n_\sigma^\phi. 
$$
Write $\Omega_1 = \pi_{\phi(\pi)\phi(\rho)}^{-1}(g)$. By the definition of the structure constant, we see that 
$\sharp \Omega_1 = p_{\phi(\pi)\phi(\rho)}^{\phi(\nu)}$. Define a map $\theta : \Omega_0 \to \Omega_1$ by 
$\theta(\widetilde{f}) = ( \phi(\widetilde{f}), \phi(\widetilde{f})^{-1}g)$. The definitions of the integers 
$n_\pi^\phi$ and $n_\rho^\phi$ yield that 
$\sharp\theta^{-1}((\alpha, \beta)) = n_\pi^\phi n_\rho^\phi$ for any $(\alpha, \beta) \in \Omega_1$. 
It turns out that   
$$
\sum_{\sigma : \phi(\sigma)=\tau}p_{\pi\rho}^\sigma n_\sigma^\phi  = \sharp \Omega_0 = n_\pi^\phi n_\rho^\phi \sharp \Omega_1 = p_{\phi(\pi)\phi(\rho)}^{\phi(\nu)} n_\pi^\phi n_\rho^\phi.
$$  
we have the result. 
\end{proof}

Lemma \ref{lem:admissible_key2} enables to us to prove the following proposition with the same argument 
as in the proof of \cite[Corollary 6.4]{F}.  

\begin{prop}\label{prop:algebra_homo}
Let  $({\mathcal D}, T)$ and $(\C, S)$ be  a finite basic schemoid and a finite quasi-schemoid 
whose underlying category is a groupoid, respectively.  
Let $\phi : (\C, S) \to ({\mathcal D}, T)$ be an admissible morphism.  
Then the function $\K(\phi) : \K(\C, S) \to \K({\mathcal D}, T)$ defined by 
$\K(\phi)(s_\pi) = n_\pi^\phi s_{\phi(\pi)}$ is an algebra homomorphism, where 
$s_\pi = \sum_{p \in \pi}p$. 
\end{prop}

\begin{rem}\label{rem:condition_P}
Let $(\C, S)$ be a quasi-schemoid which satisfies the following condition. 

\medskip
\noindent
(P) : For any $\sigma$, $\tau$ and $\mu$ in $S$, there exists at most one solution of the equation $f\circ g = h$ with  $f \in \sigma$, 
$g \in \tau$ and $h \in \mu$ if any two of $f$, $g$ and $h$ are given. 

\medskip
\noindent
Then we can prove 
Lemmas \ref{lem:admossible_key1} and \ref{lem:admissible_key2} without assuming that $\C$ is a groupoid. 
Thus,  Proposition \ref{prop:algebra_homo} remains valid if $(\C, S)$ satisfies the condition (P) 
instead of assuming that $\C$ is groupoid. 
\end{rem}

We define a category ${\mathcal B}$ to be the subcategory of $q{\bf ASmd}$ 
consisting of finite basic schemoids and admissible morphisms. Let ${\bf Alg}$ denote the category of (possibly nonunital) algebras. 
Then we have the following theorem. 

\begin{thm}
The function defined in Proposition \ref{prop:algebra_homo} gives rise to a well-defined functor 
$
\K( - ) : {\mathcal B} \to \bf{Alg}. 
$
\end{thm}
\begin{proof}
This follows from the same argument as in the proof of \cite[Corollary 6.6]{F}. 
\end{proof}


\begin{lem}\label{lem:Groupoids_addmissible}
Let $\widetilde{S}( \ ) : {\bf Gpd} \to {\bf ASmd}$ and $k : {\bf ASmd} \to q{\bf ASmd}$ 
be the functors described in Section 2. For any morphism $F : {\mathcal H} \to {\mathcal G}$ in ${\bf Gpd}$ which is injective on $ob({\mathcal H})$,  
$\phi:=\widetilde{S}(F) : \widetilde{S}({\mathcal H}) \to  \widetilde{S}({\mathcal G})$ is admissible. 
\end{lem}

\begin{proof}
We write $(\widetilde{{\mathcal H}}, S, T)$ and $(\widetilde{\G}, S', T')$ for $\widetilde{S}({\mathcal H})$ and  $\widetilde{S}({\mathcal G})$. 
We choose an object $f$ in $\widetilde{S}({\mathcal H})$, namely a morphism in ${\mathcal H}$, $\sigma_h \in S$ 
and $(u, v) \in \phi(\sigma_h)$ with $u =  \phi(f)=(\widetilde{S}(F))(f) = F(f)$. Suppose that 
$\phi(\sigma_h) \subset \sigma_k =\{(w, z) \mid w^{-1}z = k\}$ for some $k \in mor(\G)$. Then we see that $v = uk$. It follows from the definition of $
\widetilde{S}(F)$ that $F(h)= \phi(h) = k$. Since $uk = F(f)F(h)$, it follows from the injectivity of the functor $F$ on the objects that $f$ and $h$ is composable; that is, $t(h) =s(f)$. 
We have $\phi(f, fh) = (F(f), F(f)F(h)) = (u, v)$.   This completes the proof. 
\end{proof}

A schemoid extension gives an admissible morphism. More precisely,  we have the following result. 

\begin{prop}\label{prop:extension_admissible}
Let $(\E, \widetilde{S}, \widetilde{T})$ be the schemoid extension described in 
Theorem \ref{thm:extension_schemoid} 
whose underlying linear extension is of the form 
$
D_+ \to {\mathcal E} \stackrel{q}{\to} \C. 
$
Then the proper map $q$ is admissible. Moreover, one has $n_\sigma^q = \sharp D_g$ for any 
$g \in q(\sigma)$ if $C$ is basic; see Lemma \ref{lem:admossible_key1}.

\end{prop}

\begin{proof}
The result follows from the definition of a linear extension. 
\end{proof}

\begin{rem}\label{rem:basic_finite_schemoids}
Let $(\E, \widetilde{S}, \widetilde{T})$ be the same schemoid extension 
as in Theorem \ref{thm:extension_schemoid}. We see that $(\E, \widetilde{S}, \widetilde{T})$ is not in the category ${\mathcal  B}$ in general.  
In fact, the extension is unital if and only if  
$D_g$ is trivial for any $g \in mor(\C)$. 
However, Proposition \ref{prop:algebra_homo} allows one to obtain the algebra map 
$\K(q) : \K(\E, \widetilde{S}) \to \K(\C, S)$ provided $\C$ is a finite basic schemoid and $D_g$ is a finite abelian group 
for any $g \in mor (\C)$.   
\end{rem}

\begin{cor}\label{cor:isomorphism}
Under the same assumption as in Remark \ref{rem:basic_finite_schemoids}, the algebra map 
$\K(q) : \K(\E, \widetilde{S}) \to \K(\C, S)$ is an isomorphism if and only if the characteristic $ch(\K)$ of $\K$ does not divide $\sharp D_g$ for any 
$g \in mor(\C)$. In particular, if $ch(\K)$ divides $\sharp D_g$ for some $g$, then $\K(q)$ is trivial. 
\end{cor}

\begin{proof}
The fact that $\K(q)(s_\pi) = n_\pi^qs_{q(\pi)}= \sharp D_gs_{q(\pi)}$ and Theorem  \ref{thm:extension_schemoid} yield the result. 
\end{proof}



We here summarize categories mentioned above and the functors between them together with related categories. 
Let ${\mathcal S}$ be a wide subcategory of ${\bf AS}$ in the sense of French \cite{F} and 
${\mathcal A} : {\mathcal S} \to {\bf Alg}$ denote the functor defined in \cite[Corollary 6.4]{F}. 
Let ${\bf Top}$ and  ${\bf Set}^{\Delta^{op}}$ denote the category of topological spaces and that of simplicial sets, respectively.  Let ${\bf Gpd'}$ be the subcategory of ${\bf Gpd}$ consisting of the same objects 
and morphisms (functors) which are injective on the set of objects; see Lemma \ref{lem:Groupoids_addmissible}. 
Then we have a commutative diagram
$$
{\small 
\xymatrix@C17pt@R12pt{
  {\bf Gpd} \ar@<-1ex>[r]_(0.55){{\widetilde S}( \ )}^-{\simeq} 
     & (t{\bf ASmd})_0 \ar@<-1ex>[l]_-{{\widetilde R}} \ar[rr] & & 
     {\bf ASmd} \ar[r]^k & 
q{\bf ASmd} \ar@<1ex>[r]^-{U} \ar@{.>}[rdd]^(0.6){\K( \ )}
    & {\bf Cat} \ar@<1ex>[l]^-{K}  \ar@<1ex>[r]^-{N( \ )} & {\bf Set}^{\Delta^{op}} \ar@<1ex>[l]^-{c} \ar@<-1ex>[r]_-{| \ |}& 
    {\bf Top} \ar@<-1ex>[l]_-{S_*( \ )} \\ 
  {\bf Gpd'} \ar[u]^{i'}   \ar[rr]^(0.7){\widetilde{S}( \ )} & &  {\mathcal B} \ar[rru] \ar[rrrd]^{\K( \ )}&  \\ 
 {\bf Gr} \ar[u]^i \ar[r]^-{S( \ )}_-{\simeq}  \ar@/_0.9pc/[rrd]_{S( \ )} & (t{\bf AS})_0 \ar@{->}'[u][uu]_(-0.5){{j}_{(t{\bf AS})_0}} \ar[rr] & & {\bf AS} \ar[uu]_{j} 
 \ar@{.>}[rr]_(0.4){{\mathcal A}( \ )}     & & {\bf Alg}, \\
                 & &  {\mathcal S} \ar[uu]^(0.3){j_{\mathcal S}} \ar[ru] \ar[rrru]_{{\mathcal A}( \ )} & 
}
}
\eqnlabel{add-2}
$$ 
where dots arrows denote the assignments in the objects but not functors and 
vertical arrows are fully faithful functors defined by restricting the functor $j$ 
to the source categories; see Theorem  \ref{thm:functors}. 
The arrows $N( \ )$ and $c$ denote the functors induced by the nerve construction and 
the categorification (categorical realization)\cite{G-Z}, respectively; 
see also \cite{Thomason}. Moreover, $| \ |$ and $S_*( \ )$ are the realization functor and the functor induced 
by the singular simplex construction, respectively.  
Observe that, for the functors in two parallel lines, the lower arrow denotes the left adjoint for the upper one 
and that the functor $K$ is fully faithful as mentioned in Section 3. 
In a strict sense, the functor $\widetilde{S}( \ ) : {\bf Gpd'} \to {\mathcal B}$ and the correspondence 
${\mathbb K}( \ )$ from $q{\bf ASmd}$ to ${\bf Alg}$ should be restricted to the full subcategory of finite groupoids and 
to that of finite quasi-schemoids, respectively. 

We conclude this section with an important remark. 

\begin{rem}\label{rem:iso_but_not_eqivalent}
Let $(X, S)$ be an association scheme. 
Let $\E_0$ and $\E_1$ denote the trivial and non-trivial extensions of $j(X, S)\times (\mathbb{Z}/2)^\bullet$ in Example \ref{ex:blow-ups}, 
respectively.  
Corollary \ref{cor:isomorphism} implies that 
$$\K(\E_1)\cong K(\E_0) \cong \K(j(X, S)\times (\mathbb{Z}/2)^\bullet)\cong  \K(X, S).$$ 
as algebras if $ch(\K)\neq 2$, where the last one denotes the Bose-Mesner algebra of the association scheme $(X, S)$. 
We see that the schemoids $j(X, S)\times (\mathbb{Z}/2)^\bullet$ and $j(X, S)$ are not equivalent as a category; see Example \ref{ex:blow-ups}. This implies that there exist schemoids whose Bose-Mesner algebras are isomorphic to each other but not the underlying categories. 
\end{rem}

\section{How to construct (quasi-)schemoids}

In this section, we explain a way to construct a (quasi-)schemoid thickening a given association scheme. 

Let $Z = \big(z_{ij}\big)$ be a square matrix of natural numbers. 
We call the matrix $Z$ transitive if $z_{ij}, z_{jk} \geq 1$, then $z_{jk} \geq 1$.  
Let $\C$ be a finite category; that is $\sharp mor(\C) < \infty$. We can consider $ob(\C)$ an ordered set $\{i\}_{i\in ob(\C)}$. 
We define a matrix $Z=\big(z_{ij}\big)$, which is referred to as the matrix of the category $\C$, by 
$z_{ij} = \sharp \text{Hom}_\C(i, j)$. Observe that the matrix is transitive. 
We recall a result due to Berger and Leinster. 

\begin{lem}\label{lem:B-L}\cite[Lemma 4.1]{B-L}
Let $Z$ be a transitive square matrix of natural numbers whose diagonal entries are at least $2$. Then $Z$ is the matrix of a category. 
\end{lem}

We will show that the category $\C$ constructed in the proof of Lemma \ref{lem:B-L} can be endowed with 
a quasi-schemoid structure under appropriate assumption. 
To see this, we recall the construction of the category $\C$.  
Let $Z$ be an $m\times m$ matrix $\big( z_{ij} \big)$. For each pair $(i, j)$ of objects such that $z_{ij}\geq 1$, 
we choose an arrow $\phi_{ij} : i \to j$ with $\phi_{ii}\neq 1_i$ for all $i$.  Define the composite $i \stackrel{\alpha}{\to} j \stackrel{\beta}{\to} k$ by 
$\beta\circ\alpha = \phi_{ik}$ if $\alpha\neq 1$ and $\beta \neq 1$. Then we have a finite small category $\C_Z$. 
 It is readily seen that $Z$ is the matrix of the category of 
$C$. We call the subset $\{\phi_{ij}\}_{i, j \in ob(\C_Z)}$ of $mor(\C_Z)$ the {\it frame} of $\C_Z$.  
The definition of the concatenation of morphisms in $\C_Z$ above yields the following lemma. 

\begin{lem}\label{lem:frames}
In the category $\C_Z$, for a morphism $\alpha$ which is not in the frame 
$\{\phi_{ij}\}_{ij}$ of the category $\C_Z$, if $u\circ v = \alpha$, then $u=1$ or $v=1$. 
\end{lem}

\begin{prop}\label{prop:construction1} 
Let $F =\{\phi_{ij}\}_{i, j}$ be the frame of the small category $\C_Z$ constructed above. Let $1$ be the subsets 
$\{1_i\}_{i \in ob(\C_Z)}$. 
Then for any partition $\{Q_\lambda \}_{\lambda}$ of the set $mor(\C_Z) \backslash (F \cup 1)$,  
the partition of $mor(\C_Z)$
$$
\Sigma' = \{ \{\phi_{ij}\} \mid i, j \in ob(\C_Z)\}\cup 1 \cup \{Q_\lambda \}_{\lambda} 
$$ 
satisfies the concatenation axiom; see Definition \ref{defn:schemoid}. 
In consequence, $(\C_Z, \Sigma')$ is a unital quasi-schemoid.   
\end{prop}

\begin{proof}
The intersection number $p^{\{\phi_{ij}\}}_{u,v}$ is valid because $\{\phi_{ij}\}$ is a set of single element. 
Write $\widetilde{J}=  J\backslash (\{\phi_{ii}\}\cup 1)$. For $L = Q_\lambda, \widetilde{J}$, 
the usual argument shows that
 $p^{L}_{uv} = 0$ for any $u, v \in \{\{\phi_{ij}\} \mid i, j\}\cup \{Q_\lambda \}_\lambda\cup \widetilde{J}$, 
 $p^{L}_{1u} = \delta_{Lu}= p^{L}_{u1}$ and $p^{L}_{11}=0$. 
It is readily seen that 
 $$
 p^1_{uv} = 
 \left\{
\begin{array}{l}
1 \ \ \ \text{if} \  u=1=v    \\
0  \ \ \text{otherwise} 
\end{array}
\right.
$$
This completes the proof.  
\end{proof}

\begin{prop}\label{prop:construction2}
Let $S$ be a partition of the frame $\{\phi_{ij}\}_{i, j}$ of the finite category $\C_Z$ constructed above, namely 
$\{\phi_{ij}\}_{i, j}= \coprod_{\sigma \in S}\sigma$. Suppose that 
$S$ satisfies the concatenation axiom and that the condition (i) in Definition \ref{defn:association_schemoids} 
for $S$ holds. Assume further that $z_{ij}=z_{i'j'}=:z_\sigma$ 
for any $\sigma \in S$, $\phi_{ij}, \phi_{i'j'} \in \sigma$. 
We define $\widetilde{\sigma} = \{\phi_{ij}^\lambda \}_{\phi_{ij}\in \sigma, \lambda = 1, .., z_{ij}-1}$ for $\sigma \in S$.  
Then 
$$\Sigma = \{1\} \cup S \cup \widetilde{S} $$
satisfies the concatenation axiom and hence the pair $(\C_Z, \Sigma)$ is a unital quasi-schemoid, 
where $1 = \{1_i \}_{i \in ob(\C_Z)}$ and 
$\widetilde{S} = \{\widetilde{\sigma} \}_{\sigma \in S, \widetilde{\sigma}\neq \phi}$. 
\end{prop}

\begin{proof} We first observe that for any $f, g \in mor(\C)$, $f\circ g \in \coprod_{\sigma \in S}\sigma$ if 
$f$ and $g$ are composable and $f\circ g \neq 1_i$ for some $i$. 
Moreover, it is immediate that if $f \circ g = 1_i$ then $f = g= 1_i$. Thus we have 
i) $p^{1}_{11} = 1$, 
ii) $p^1_{uv} = 0$  for $u, v \in S \cup \widetilde{S}$, 
iii) $p^{\widetilde{\sigma}}_{11} = 0$, 
iv) $p^{\widetilde{\sigma}}_{1\widetilde{\tau}} = \delta_{\widetilde{\sigma}\widetilde{\tau}}= 
p^{\widetilde{\sigma}}_{\widetilde{\tau} 1}$, 
v) $p^{\widetilde{\sigma}}_{uv} = 0$ for $u, v \in S \cup \widetilde{S}$, 
vi) $p^{\widetilde{\sigma}}_{1{\tau}} = 0 =  p^{\widetilde{\sigma}}_{{\tau} 1}$, 
vii) $p^\sigma_{11} = 0$, 
viii) $p^\sigma_{1\widetilde{\tau}} = 0 = p^\sigma_{\widetilde{\tau}1}$ and 
ix) $p^\sigma_{1\tau} = \delta_{\sigma\tau} = p^\sigma_{\tau 1}$. 
The assumption on the entries $z_{ij}$ enables us to obtain the following equality; see the figure below. \\
x) $p^\sigma_{\tau\widetilde{\mu}} = p^\sigma_{\tau\mu}\cdot (z_\mu -1)$, \\
xi) $p^\sigma_{\widetilde{\mu}\tau} = (z_\mu-1)\cdot p^\sigma_{\mu\tau} $,  \\
xii) $p^\sigma_{\widetilde{\mu}\widetilde{\tau}} = (z_\mu -1)\cdot p^\sigma_{\mu\tau} \cdot (z_\tau-1)$. 
$$
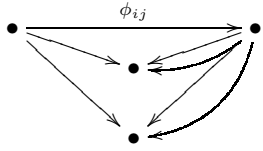
\xymatrix@C35pt@R5pt{
\bullet \ar[rr]^{\phi_{ij}} \ar[dr] \ar[dddr]  & & \bullet \ar[dl] \ar[dddl] \ar@/^/[dl] \ar@/^1pc/[dddl] \\
 & \bullet &  \\
 &   & \\
 &  \bullet  &}
$$
It is immediate that the condition (i) in Definition \ref{defn:schemoid} holds for the partition $\Sigma$. 
We have the result. 
\end{proof}

Let $(X, P=\{P_l\}_{l= 0, .., s})$ be an association scheme with $X = \{1, ..., m\}$, where $P_0 = \{(i,i) \mid i \in X\}$. 
Let $R_l$ denote the adjacency matrix associated with $P_l$; that is, 
the $(i, j)$ entry $R_l(i,j)$ of $R_l$ is defined by 
$$
R_l(i,j) = 
 \left\{
\begin{array}{l}
1 \ \ \ \text{if} \  (i, j) \in P_l   \\
0  \ \ \text{otherwise} 
\end{array}
\right.
$$
Thickening the given association scheme, we can obtain an association schemoids. 

\begin{thm}\label{thm:construction3}
With the same notation as above, for positive integers $z_0, ..., z_s$, 
we define an $m\times m$ matrix $Z$ by $$Z = z_0 R_0 + z_1 R_1 + \cdots + z_s R_s + \text{\em diag}(1, 1, ..., 1).$$
Let $S=\{\sigma_l\}_{l=0, 1, .., s}$ be a partition of the frame $\{\phi_{ij}\}_{ij}$ of the category $\C_Z$, where 
$\sigma_0 = \{\phi_{ii} \mid i = 1, ..., m\} $ and $\sigma_l =\{\phi_{ij} \mid (j, i)\in P_l \}$. Then $S$ satisfies the concatenation axiom.
In consequence, the pair $(\C_Z, \Sigma)$ with $\Sigma$ defined in Proposition \ref{prop:construction2} is 
a unital quasi-schemoid. Moreover, if $z_0 = \cdots = z_s$, then $(\C_Z, \Sigma)$ admits a unital schemoid structure.  
\end{thm}

\begin{proof} Since $(X, P)$ is an association scheme, we see that $S$ satisfies the concatenation axiom and that 
the condition (i) in Definition \ref{defn:association_schemoids} holds.  
Proposition \ref{prop:construction2} implies that $(\C_Z, \Sigma)$ is a unital quasi-schemoid. 

Suppose that  $z_0 = \cdots = z_s$. For any objects $i, j$, let  $M(i, j)$ be the subset of $mor(\C_Z)$ 
consisting of morphisms $f$ which satisfies the condition that $s(f)=i$, $t(f)=j$ and $f \neq \phi_{ij}$. Then there exists a bijection 
$\theta : M(i, j) \to M(j, i)$. With the bijection, we define a contravariant functor 
$T : \C_Z \to \C_Z$ with $T^2=1_{\C_Z}$ by $T(\phi_{ij}) = \phi_{ji}$ 
and for $f$ which is not in the frame, by $T(f) = \theta(f)$. 
We then have an association schemoid $(\C_Z, \Sigma, T)$.     
\end{proof}

For an association scheme $(X, P)$, 
we may write ${\mathcal S}{\C}_{z_0, ..., z_s}(X, P)$ for the quasi-schemoid $(\C_Z, \Sigma)$ constructed via the 
procedure in Theorem \ref{thm:construction3}. In what follows, we denote  by ${\mathcal S}{\C}_{z}(X, P)$ 
the induced association schemoid  ${\mathcal S}{\C}_{z, ..., z}(X, P)$.

\begin{prop}\label{prop:Phi}
Let $(X, P)$ be an association scheme. With the same notation as in Proposition \ref{prop:construction2}, 
we define a functor $\Phi : {\mathcal S}{\C}_{z_0, ..., z_s}(X, P) \to j(X, P)$ 
by $\Phi(\phi_{ij}^\lambda) = \phi_{ij}$  for any $\phi_{ij}^\lambda \in S \cup \widetilde{S}$, where $\phi_{ij}^0=\phi_{ij}$ and 
$X$ is considered as a directed complete graph with $\phi_{ij}$'s as edges  provided 
$\phi_{ii} = 1_i$. Then $\Phi$ is an admissible map in the sense of Definition \ref{defn:admissible_maps}. 
\end{prop}

\begin{proof}
The result follows from the construction of the quasi-schemoid $(\C_Z, \Sigma)$. 
\end{proof}

\begin{rem} Under the same notation as in Proposition \ref{prop:construction2}, 
we consider an equation of the form $\phi_{kj}^{\e'}\circ\phi_{ik}^\e= \phi_{ij}$, where 
$\e' = 0, 1$, $\e = 0, 1$ and $\phi_{lm}^0 = \phi_{lm}$. For given elements $\phi_{ik}^\e$ and $\phi_{ij}$, a solution $\phi_{kj}^{\e'}$ 
of the equation is exactly one in an element of $\{1 \}$, $S$ or $\widetilde{S}$. Moreover, 
we see that if the equation $f\circ g = \phi_{ij}^1$ has a solution, then  
one of $f$ and $g$ should be the identity. This yields that the association schemoid ${\mathcal S}{\C}_{z}(X, P)$ satisfies the condition (P) 
in Remark \ref{rem:condition_P} if $z = 1$ or $2$. 
\end{rem}

\begin{rem} \label{rem:A&K} Let $\Phi : {\mathcal S}{\C}_{1}(X, P) \to j(X, P)$ be the admissible map in Proposition \ref{prop:Phi}. 
Then the induced map $\K(\Phi) : \K({\mathcal S}{\C}_{1}(X, P)) \to {\mathcal A}(X, P)$  is not an isomorphism.
In fact, we see that 
$\dim {\mathcal A}(X, P) + 1 = \dim \K({\mathcal S}{\C}_{1}(X, P))$. 
\end{rem}

\begin{rem} We recall $U : q{\bf ASmd} \to {\bf Cat}$ the forgetful functor described in the diagram (6.1).
Then the classifying space $B\C$ of the category $\C = U{\mathcal S}{\C}_{z}(X, P)$ is a contractible. 
To see this, we choose an object $i$ of $\C$ and consider the functor 
$\pi : \C \to \bullet$ and $\iota : \bullet \to \C$, where $\bullet$ denotes the trivial category with the one object $\ast$ and $\iota(\ast) =i$.  
It is immediate that $\pi\circ \iota = 1_\bullet$. We may define a natural transformation $\eta : \iota\circ \pi \to 1_{\C}$ by 
$\eta(k) = \phi_{ik}$ for $k \in ob(\C)$. In fact, for any $\phi_{kl}^\lambda : k \to l$ in $\C$, we see that 
$$
\eta(l)\circ ((\iota\circ \pi) (\phi_{kl}^\lambda))=\phi_{il}\circ 1_i = \phi_{il} = \phi_{kl}^\lambda\circ \phi_{ik} = \phi_{kl}^\lambda \circ \eta(k).
$$
Then we see that $B(\pi)\circ B(\iota) = 1_{B\bullet}$ and $B(\iota)\circ B(\pi)\simeq 1_{B\C}$. 
This implies that $B\C$ is homotopy equivalent to the space of a point and hence to $B(Uj(X, P))$; see the comments described before 
Corollary \ref{cor:trivial_extensions}. On the other hand, 
Remark \ref{rem:A&K} states that $\K j(X, S)$ is not isomorphic to $\K{\mathcal S}{\C}_{1}(X, P)$. 

It is important to recall Corollary \ref{cor:isomorphism} and the results in Examples \ref{ex:contractible_spaces} and 
\ref{ex:blow-ups}. In some case, a linear extension $D_+ \to (\C, \widetilde{S}, \widetilde{T}) \stackrel{q}{\to} j(X, P)$ 
gives rise to an isomorphism $\K(q) :  \K(\C, \widetilde{S}) \stackrel{\cong}{\to} \K j(X,P)= {\mathcal A}(X, P)$ while the classifying space 
$B(U(\C, \widetilde{S}))$ is not homotopy equivalent to $B(Uj(X, P))$. 
\end{rem}

The following theorem ensures that the construction in Theorem \ref{thm:construction3} is functorial. 

\begin{prop}
Let $z$ be a positive integer. Then the correspondence sending an association scheme $(X, P)$ to the association 
schemoid ${\mathcal S}{\C}_z(X, P)$ gives rise to a functor 
${\mathcal S}{\C}_z( \ ) : {\bf AS} \to {\bf ASmd}$. Moreover, association schemes $(X, P)$ and $(X', P')$ 
are isomorphic if and only if so are schemoids  ${\mathcal S}{\C}_z(X, P)$ and ${\mathcal S}{\C}_z(X', P')$.  
\end{prop}

\begin{proof} 
Let $\{\phi_{ij}\}$ and $\{\psi_{lm}\}$ be the frames of ${\mathcal S}{\C}_z(X, P)$ and of ${\mathcal S}{\C}_z(X', P')$, respectively. 
Let  $f : (X, P) \to (X', P')$ be a morphism of association schemes. With the same notation as in Proposition \ref{prop:construction2}, 
we define ${\mathcal S}{\C}_z(f)(\phi_{ij}^\lambda) = \psi_{lm}^\lambda$ for $\lambda = 0, ...., z$ if 
$f(\phi_{ij}) = f(\psi_{lm})$, where  $\phi_{ij}^0 = \phi_{ij}$ and   $\psi_{lm}^0=\psi_{lm}$. Then ${\mathcal S}{\C}_z( \ )$ is a well-defined functor.  
This yields the first part of the assertion.  
As for the latter half, the "only if part" is immediate. 
Suppose that $F : {\mathcal S}{\C}_z(X, P) \to {\mathcal S}{\C}_z(X', P')$ is an isomorphism of schemoids. Since 
$F(\phi_{ij}) = F(\phi_{jj}\circ \phi_{ij}) = F(\phi_{jj})\circ F(\phi_{ij})$, it follows from Lemma \ref{lem:frames} that 
$F$ sends elements in the frame of ${\mathcal S}{\C}_z(X, P)$ to those of ${\mathcal S}{\C}_z(X', P')$. 
This completes the proof.   
\end{proof}

\begin{ex}
Let $H(2, 2)$ be the Hamming scheme of type $(2,2)$; that is, 
$H(2, 2) = ( \{0, 1\}^2, \{T_0, T_1, T_2\} )$, 
where $T_i$ denotes the set of the pair of words with the Hamming metric $i$; see Section 8.
Pictorially, we have  
$$
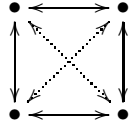
\xymatrix@C30pt@R30pt{
{\bullet}  \ar@{<->}[r] \ar@{<.>}[rd]  \ar@{<->}[d] &  {\bullet}  \ar@{<->}[d]\\
{\bullet} \ar@{<->}[r] \ar@{<.>}[ru] &  {\bullet}
}
$$
where white arrows from a vertex to itself,  the black arrows and dots arrows are 
in $T_0$, $T_1$ and $T_2$, respectively. 
By applying Theorem \ref{thm:construction3}, we obtain an association quasi-schemoid $(\C_Z, \Sigma)$ 
constructed by the $4\times 4$ matrix with data of the partition 
$$
Z= n_0R_0 + n_1R_1 + n_2R_2 +\text{diag}(1, 1, 1, 1) =\left(
\begin{array}{cccc}
n_0+1& n_1 & n_1 & n_2 \\
n_1 & n_0+1 & n_2 & n_1 \\
n_1 & n_2 & n_0+1 & n_1 \\
n_2 & n_1 & n_1 & n_0+1
\end{array}
\right). 
$$
The schemoid $(\C_Z, \Sigma)$ can be represented by the picture 
$$
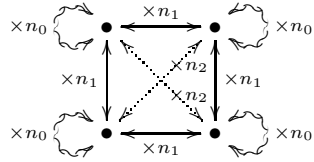
\xymatrix@C30pt@R30pt{
{\bullet}  \ar@(ul,dl)@{<~>}[]_{\times n_0} \ar@{<->}[r]^{\times n_1} \ar@{<.>}[rd]^{\times n_2}  
\ar@{<->}[d]_{\times n_1} &  {\bullet}  \ar@(ur,dr)@{<~>}[]^{\times n_0} \ar@{<->}[d]^{\times n_1} \\
{\bullet} \ar@(ul,dl)@{<~>}[]_{\times n_0} \ar@{<->}[r]_{\times n_1} \ar@{<.>}[ru]_{\times n_2} &  
  {\bullet} \ar@(ur,dr)@{<~>}[]^{\times n_0}
}
$$
with $4$ white identities. 
\end{ex}

\medskip
\noindent
{\it Acknowledgements.} The first author thanks Akihide Hanaki, Akihiko Hida and Mitsugu Hirasaka 
for precious discussions on association schemes, their representations, extensions and related topics. 
The authors are grateful to Yu Maekawa for explaining important properties of structure constants. They also thank the referee 
for valuable suggestions to revise a previous version of this paper.

\section{Appendix: A schemoid and a toy model for a network}
In this section, we relate a schemoid with a toy model for a network. 

Let ${\mathbb F}$ be a set of $q$ elements and $X$ the product set ${\mathbb F}^n$. 
Then we obtain the Hamming scheme $H(n, q) = (X, \{R_k\}_{k=1, ..., n})$ which is one of crucial association schemes; 
that is, with the Hamming metric 
$\partial (x, y)= \sharp\{i \mid  x_i \neq y_i\}$ for $x= (x_i)$ and $y = (y_i)$, $R_k$ is defined by 
$
R_k :=\{(x, y) \in X\times X \mid \partial (x, y) =k\}. 
$    

A main problem in coding theory is to estimate the maximal size of any subset (code) $C$ of $X$ such that no 
two elements in $C$ have the Hamming metric less than a given value. We here refer to elements in $X$ as passwords. 
Though the problem considering the maximal size of such a set originates in the study of error correcting code, 
the importance of the problem also comes from other reason. In fact, 
when an {\it administrator} distributes passwords among {\it users}, 
he or she should need more passwords which are different from one another with appropriate large metric.  
Someone of users might use a wrong password $w$ which differs from correct one a little bit. 
If $w$ is a password of another someone, that gives rise to an inadvisable circumstance. 

It is significant to mention that upper-bounds of the maximal size of such a code are considered by using 
the Bose-Mesner algebra of an association scheme; see \cite{D}. 

After determining a maximal code $C$, a system of a community begins to work.   
We can consider the whole Hamming scheme or a code $C$ a complete graph with passwords as vertices.    
If one changes the password to some other one along edges,  
then {\it directions} would be needed.  It is natural to regard such directed edges as morphisms 
connecting the passwords (vertices, objects). Thus a categorical notion appears in the system. 

In order to explain such notion a little more, 
a toy model which an administrator constructs is given here using undefined terminology. 
For vertices $x$, $y$ and $z$ in a directed graph, the concatenation of directed edges (morphisms) 
$(z, x)$ and $(x, y)$ is defined by $(z, x)\circ (x, y) = (z, y)$. 
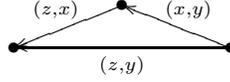
$$
\xymatrix@C35pt@R10pt{
& *=0{\bullet}  \ar@{->}[dl]_{(z, x)}   & \\
*=0{\bullet} & & *=0{\bullet}\ar@{->}[ul]_{(x, y)} \ar@{->}[ll]^{(z, y)}
}
$$
The administrator might establish a rule which prohibits users from changing freely the passwords via edges.   
If he or she only admits changes of the passwords through {\it odd number} of directed edges, then 
the equality $(z, x)\circ (x, y) = (z, y)$ above is inconvenient. 
Thus a blow-up of the set of edges is needed. We attach some edges to the graph as the diagram below. 
$$
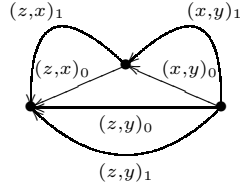
\xymatrix@C30pt@R10pt{
& *=0{\bullet}  \ar@{->}[dl]_{(z, x)_0}  \ar@{->}@/_2pc/[dl]_{(z, x)_1}  &  \\
*=0{\bullet} & & *=0{\bullet}\ar@{->}[ul]_{(x, y)_0} \ar@{->}@/_2pc/[ul]_{(x, y)_1} \ar@{->}[ll]^{(z, y)_0}
\ar@{->}@/^1.5pc/[ll]^{(z, y)_1}    }
$$
Here $0,1 \in {\mathbb Z}/2$.  
Then a natural concatenation is given by 
\begin{equation}
\label{eq_1}
(z, x)_1\circ (x, y)_1 = (z, y)_{1+1} = (z, y)_0 \neq (z, y)_1. 
\end{equation}
Using only the directed edges of the form $(*, *)_1$, a system of a network that he or she requires may be constructed.  
Indeed, such a system, which has both structures of an association scheme and a small category, seems to be a schemoid. 
Observe that the concatenation law (\ref{eq_1}) is realized in a linear extension of a category due to 
Baues and Wirsching \cite{B-W}. In fact,  let $\E_\eta$ be  the extension in Remark \ref{rem:iso_but_not_eqivalent} 
for $\eta = 0, 1$. Then we see that 
$$
(f, a, a') \circ (g, b, b') = (f\circ g, a+b, \eta\Delta(a, b) + a' + b'), 
$$
where $\Delta( \ , \ ) : {\mathbb Z}/2 \times {\mathbb Z}/2 \to {\mathbb Z}/2$ is a $2$-cycle defined by 
$\Delta(1,1)= 1$ and $\Delta(s, t) = 0$ for $(s,t) \neq (1,1)$. 

The administrator might use the extension $\E_0$ and morphisms of the form  
$(*, 0, 1)$ only  to construct a system because 
$
(f, 0, 1) \circ (g, 0, 1) = (f\circ g, 0, 0)\neq (f\circ g, 0, 1). 
$
If he or she uses the extension  $\E_1$ and morphisms of the same form $(*, 0, 1)$, then one has 
$
(f, 0, 1) \circ (g, 0, 1) = (f\circ g, \Delta(1, 1)+ 1+1)=(f\circ g, 0, 1)
$
and hence all users can change their passwords freely.


\end{document}